\documentclass{article}
\usepackage{amsthm,amsmath,amssymb,amsfonts}
\usepackage{latexsym,enumerate,vmargin,relsize,mathabx,enumitem}
\usepackage{array}
\usepackage[usenames,dvipsnames]{color}
\usepackage{hyperref,url}
\usepackage[normalem]{ulem}
\usepackage{tikz}
\usetikzlibrary{arrows}
\usetikzlibrary{decorations.markings}
\usetikzlibrary{calc}

\setmarginsrb{30mm}{20mm}{30mm}{20mm}{0pt}{10mm}{0pt}{10mm}

\def\N{\mathbb{N}}

\def\Z{\mathbb{Z}}
\def\R{\mathbb{R}}

\def\bold{\mathbf}

\newtheorem{thm}{Theorem}
\newtheorem*{thm*}{Theorem}

\newtheorem*{claim*}{Claim}

\newtheorem*{dfn*}{Definition}
\newtheorem{lemma}[thm]{Lemma}
\newtheorem*{lemma*}{Lemma}
\newtheorem{prop}[thm]{Proposition}
\newtheorem*{prop*}{Proposition}
\newtheorem{cor}[thm]{Corollary}
\newtheorem*{cor*}{Corollary}
\newtheorem{conj}[thm]{Conjecture}
\newtheorem*{conj*}{Conjecture}

\newtheorem*{quest*}{Question}

\newtheorem*{exer*}{Exercise}
\newtheorem{example}[thm]{Example}
\newtheorem*{example*}{Example}

\newtheorem*{examples*}{Examples}

\newtheorem*{formula*}{Formula}
\theoremstyle{remark}

\newtheorem*{rmk*}{Remark}

\newtheorem*{rmks*}{Remarks}

\title{Invariants of Random Knots and Links}
\author{
Chaim Even-Zohar 
\thanks{Department of Mathematics, Hebrew University, Jerusalem 91904, Israel. \href{mailto:chaim.evenzohar@mail.huji.ac.il}{chaim.evenzohar@mail.huji.ac.il}}
\and 
Joel Hass 
\thanks{Department of Mathematics, University of California, Davis, California 95616. \href{mailto:hass@math.ucdavis.edu}{hass@math.ucdavis.edu}}
\and 
Nati Linial 
\thanks{Department of Computer Science, Hebrew University, Jerusalem 91904, Israel. \href{mailto:nati@cs.huji.ac.il}{nati@cs.huji.ac.il}}
\and
Tahl Nowik
\thanks{Department of Mathematics, Bar-Ilan University, Ramat-Gan 5290002, Israel. \href{mailto:tahl@math.biu.ac.il}{tahl@math.biu.ac.il}}
\and
\thanks{This project was supported by BSF grant 2012188.}
}

\begin{document}

\maketitle

\begin{abstract}
We study random knots and links in $\R^3$ using the Petaluma model, which is based on the petal projections developed in~\cite{adams2012knot}. In this model we obtain a formula for the limiting distribution of the linking number of a random two-component link. We also obtain formulas for the expectations and the higher moments of the Casson invariant and the order-3 knot invariant $v_3$. These are the first precise formulas given for the distributions and higher moments of invariants in any model for random knots or links. We also use numerical computation to compare these to other random knot and link models, such as those based on grid diagrams. 

\end{abstract}

\begin{center}
\tikz[thick]{\foreach \angle in {16.36363636, 49.09090909, ..., 360} \draw (0,0) .. controls +(\angle:1.8) and +(\angle+16.36363636:1.8) .. (0,0);
\pgfresetboundingbox \clip (-1.5,-1.00) rectangle (1.5,1.00);}
\end{center}

\setcounter{section}{-1}
\section{Introduction}\label{introsect}

In this paper we study the distribution of finite type invariants of random knots and links. Our purpose is to investigate properties of {\em typical} knots, avoiding biases caused by focusing attention on a limited set of commonly studied examples. While tables of knots with up to 16 crossings have been compiled \cite{hoste1998first}, and much is understood about infinite classes of knots, such as torus and alternating knots, we suspect that our view of the collection of all knots is distorted by the choices that simplicity and availability have given us. We have little knowledge of the distribution of knot invariants such as the Jones polynomial, or the linking number, among highly complicated knots and links. Studying a model of {\em random knots} allows us to probe for typical behavior beyond the familiar classes. As we elaborate below, the spectacular success of the probabilistic method in combinatorics makes us hopeful that it has much to offer in topology as well.

A variety of models for random knots and links have been studied by physicists and biologists, as well as mathematicians. Common models are based on random 4-valent planar graphs with randomly assigned crossings, random diagrams on the integer grid in $\R^2$, Gaussian random polygons~\cite{buck1994random, arsuaga2007linking, micheletti2011polymers}, and random walks on lattices in $\R^3$ \cite{sumners1988knots, pippenger1989knots}. While many interesting numerical studies have been performed, and interesting results obtained in these models, there have been few rigorous derivations of associated statistical measures. 

In this paper we study a model of random knots and links called the {\em Petaluma model}, based on the representation of knots and links as petal diagrams that was introduced by Adams and studied in \cite{adams2012knot}. The Petaluma model has the advantage of being both universal, in that it represents all knots and links, and combinatorially simple, so that knots have simple descriptions in terms of a single permutation. We obtain here what appears to be the first precise formulas in any random model for the distributions of knot and link invariants. 

We first derive a formula for the limiting distribution of the linking number of a random two component link. This is shown to have an unexpected connection to a distribution previously studied by physicists in another context. As it turns out, the distribution of the linking number is identical to that of the signed area enclosed by a random path on the integer lattice in the plane. This has a physical interpretation as the flux of a vector field through a random planar curve. We develop a variation of the approach of Mingo and Nica to a closely related problem on signed area~\cite{mingo1998distribution}, in order to analyze this model and to obtain the linking number distribution. 

We then study the distributions of the two simplest knot invariants of finite type, namely the order-2 Casson invariant $c_2$ and the order-3 invariant $v_3$, associated to the Jones polynomial. We are able to find expressions for the expectation, variance and higher moments of these two invariants. We present these results after describing our model for random knots and links and reviewing the construction of finite-type invariants.

\subsection*{Knots and Petal Diagrams}

A \emph{knot} is a simple closed curve in $\R^3$, up to equivalence generated by an isotopy of $\R^3$~\cite{adams1994knot}, while a \emph{link} is a disjoint union of simple closed curves, with the same equivalence. The curves can be taken to be either smooth or piecewise-linear (polygonal). Knots and links are commonly represented by {\em diagrams}, which are projections of a knot or link to the plane in which a finite number of points have two preimages, and each such \emph{crossing point} is marked to indicate which point lies above the other in $\R^3$. A diagram suffices to recover a knot or link up to an isotopy of $\R^3$.

Adams et al.\ showed that an embedding of a knot or link in $\R^3$ can be chosen so that its projection has a single crossing, though the multiplicity with which the knot projects to this crossing is now allowed to be larger than two \cite{adams2012knot}. Furthermore, in the case of knots, the projected arcs can be arranged so that they trace out a rose-like curve. A \emph{petal diagram} is a planar curve, comprised of $2n+1$ straight segments crossing at a single point, and arcs connecting consecutive pairs of segment tips. This creates $2n+1$ loops with disjoint interiors, called {\em petals}. Figure~\ref{fig_petal1} shows a petal diagram with $9$ petals. 

Along with its projection, a petal diagram comes with information on how to construct a knot in $\R^3$ that projects to the diagram. The additional information specifies the height of the arcs passing above the single crossing. The ordering of these heights is specified by a permutation $\pi\in S_{2n+1}$, with $\pi(i)$ giving the knot's height as it passes over the center for the $i$th time. This representation is \emph{universal}, so that all knots are realized by some petal diagram \cite[Theorem 1]{adams2012knot}. Each permutation determines a knot, and in the {\em Petaluma model} we define a random knot $K_{2n+1}(\pi)$ to be a knot with a $2n+1$ petal diagram and permutation $\pi \in S_{2n+1}$, drawn uniformly at random.

\begin{figure}[htb]
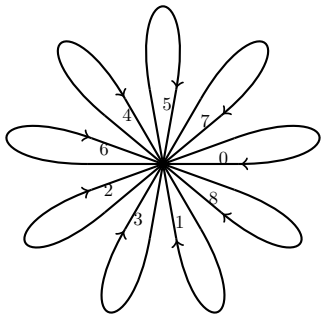

\begin{center}
\tikz[thick,decoration={markings,mark=at position 0.05 with {\arrow{<}}}]{
\foreach \angle in {0, 40, ..., 320} \draw[postaction={decorate}] (\angle:1) .. controls +(\angle:1.5) and +(\angle+20:1.5) .. (\angle+20:1);
\foreach \angle in {0, 40, ..., 320} \draw (\angle:1) -- (\angle:-1);
\foreach \angle/\num in {0/0, 80/1, 160/2, 120/3, 240/4, 280/5, 200/6, 320/7, 40/8} \node[scale=0.7] at (-\angle+6:0.8) {$\num$};
\pgfresetboundingbox \clip (-2.25,-2.00) rectangle (2.25,2.50);}
\end{center}
\caption{A petal diagram with $9$ petals and associated permutation $(0,6,8,4,1,5,3,7,2)$.}
\label{fig_petal1}
\end{figure}

This construction extends to links. A \emph{two-component petal diagram} consists of two planar curves, each of which transversely passes $2n$ times through a single point, as shown in Figure~\ref{fig_petal2} for $n=3$. Note that this diagram is not composed of two standard petal diagrams, as its restriction to each component is a \emph{pre-petal diagram}~\cite{adams2012knot}, with one big loop whose interior contains the other loops. However, the transition between pre-petal and petal diagrams is immediate, and this diagram is the closest to a petal diagram that one can get for links having more than one component.

As with a knot, a two-component link is uniquely determined by a permutation $\pi \in S_{4n}$. The strands of the first component pass above the crossing at heights $\pi(1),\ldots,\pi(2n)$ and the strands of the second at heights $\pi(2n+1),\ldots,\pi(4n)$. This gives a universal model for two-component links~\cite[Theorem 2]{adams2012knot}, and the Petaluma model for a random two-component link $L_{4n}(\pi)$, is obtained by drawing $\pi$ uniformly at random from $S_{4n}$. This model can be adjusted to allow for unequal numbers of petals in the two components, or a higher number of components.

\begin{figure}[htb]
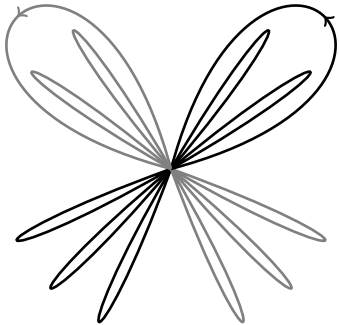

\begin{center}
\tikz[thick,decoration={markings,mark=at position 0.5 with {\arrow{>}}}]{
\foreach \angle/\color in {30/black,50/black,120/gray,140/gray,200/black,220/black,240/black,290/gray,310/gray,330/gray} \draw[color=\color,line width=1] (0,0) .. controls +(\angle:3) and +(\angle+10:3) .. (0,0);
\foreach \angle/\color in {15/black,105/gray} \draw[postaction={decorate},color=\color,line width=1] (0,0) .. controls +(\angle:4.4) and +(\angle+60:4.4) .. (0,0);
\pgfresetboundingbox \clip (-2.25,-1.5) rectangle (2.25,2.5);
}
\end{center}
\caption{A two-component petal diagram with $12$ petals.}
\label{fig_petal2}
\end{figure}

In one of its key properties, the Petaluma model differs from other models of random knots such as closed random walks in $\R^3$. The typical step length in the Petaluma model is of the same order as the diameter of the whole knot. Models which take substantially shorter steps tend to create small local entanglement which significantly affect the nature of the generated knots.

A {\em knot invariant} associates to a curve a quantity, such as a real number, a polynomial, or a group, that depends only on the knot type of the curve. Many invariants have been introduced to help in understanding the structure of knots, including the knot group, knot polynomials, the knot genus, the bridge number and the crossing number. Link invariants are similarly defined. 

An important class of knot and link invariants, called {\em finite type invariants}, were introduced by Vassiliev~\cite{vassiliev1988cohomology}, and have since been extensively studied~\cite{chmutov2012introduction}. Many knot invariants are finite type, including the coefficients of the Conway and of the modified Jones polynomials. The first two non-trivial finite type knot invariants, $c_2(K)$ and $v_3(K)$ are determined by coefficients of these polynomials. Two-component links admit a non-trivial order-1 invariant, the linking number $lk(L)$. A description of finite type invariants appears at the start of Section~\ref{linksect}.

\subsection*{Results}

We study the behavior of finite type invariants of knots and links in the Petaluma model. We view a knot invariant as a random variable on the set of all diagrams with $2n+1$ petals, and ask for its distribution and for its asymptotic growth as $n\to\infty$.

Recall that the $k$th {\em moment} of a random variable $X$ is the expected value $E[X^k]$. The moments of an invariant give a concrete indication of its value on a randomly sampled knot or link. To understand the distribution of an invariant as $n \to \infty$ we must determine how to normalize it as $n$ grows. The following theorems determine the order of growth of the finite type invariants $lk$, $c_2$, and $v_3$. 

\begin{thm}\label{moments1}
$E[(lk(L_{4n}))^k]$ is a polynomial in $n$ of degree $\leq k$.
\end{thm}

\begin{thm}\label{moments2}
$E[(c_2(K_{2n+1}))^k]$ is a polynomial in $n$ of degree $2k$.
\end{thm}

\begin{thm}\label{moments3}
$E[(v_3(K_{2n+1}))^k]$ is a polynomial in $n$ of degree $\leq 3k$.
\end{thm}

\begin{rmk*}
In Theorems~\ref{moments1} and~\ref{moments3} there is equality for $k$ even, while the odd moments are $0$.
\end{rmk*}

We also determine the leading term of the $E[c_2^k]$ polynomial. This yields the limits of the moments of the normalized invariant $c_2/n^2$. For $k=1,2,3$ we find $E[c_2^k/n^{2k}] \xrightarrow{n \to \infty} 1/24$, $7/960$, and $5119/2419200$ respectively. Similarly we obtain the limiting variance $E[v_3^2/n^6] \xrightarrow{n \to \infty} 4649/2721600$.

In the case of the linking number of two-component links, we can do more. We exactly describe the limiting distribution of the properly normalized first order invariant $lk(L_{4n})/n$ as~$n \to \nobreak\infty$.

\begin{thm}\label{linking}
The limiting probability distribution of the normalized linking number $ lk(L_{4n})/4n$ is given by 
$$ P\left[\alpha < \frac{lk(L_{4n})}{4n} < \beta \right] \;\xrightarrow{\;n\rightarrow\infty\;}\;
\int^{\beta}_{\alpha} \frac{\pi}{\cosh^2(2 \pi x)}dx \;=\; \frac{\tanh(2 \pi \beta) - \tanh(2 \pi \alpha)}{2} \;\;.$$
\end{thm}

Theorem~\ref{linking} resolves a difficulty that was encountered in the uniform random polygon model~\cite{arsuaga2007linking}. In fact, it gives the first explicit description of the asymptotic probability distribution for any knot or link invariant. Our proof of Theorems~\ref{moments1} and~\ref{linking} is an adaptation and simplification of Mingo and Nica's study of the area enclosed by a random curve~\cite{mingo1998distribution}.

Theorem~\ref{moments2} shows that $c_2(K_{2n+1})$ typically grows as $n^2$, while Theorem~\ref{moments3} show that $v_3(K_{2n+1})$ grows as $n^3$. In combination with Theorem~\ref{linking}, these results naturally suggest the following conjecture:

\begin{conj}\label{limitdist}
Let $v_m$ be a knot invariant of order $m > 0$. Then $v_m(K_{2n+1})/n^m$ weakly converges to a limit distribution as $n \rightarrow \infty$.
\end{conj}

In Section~\ref{cassonsect} we discuss computational evidence for the existence of a limit distribution for $c_2/n^2$. It is illustrative to note the {\em extremely atypical} behavior of the above-mentioned torus knots. Whereas $c_2$ is typically quadratic, it is of order $n^4$ for $(n,n+1)$ torus knots.

\begin{rmk*}
By considering the \emph{blackboard framing} of petal diagrams we obtain a universal model of~\emph{framed knots}~\cite[p. 17]{chmutov2012introduction}. This allows us to analyze finite type invariants of random framed knots, e.g., the \emph{writhe}~\cite[p. 6]{chmutov2012introduction}. This analysis can be carried out using methods similar to those employed for knots. We will not discuss the framed case in this paper.
\end{rmk*}

\subsection*{The Probabilistic Method}

The analysis of knot invariants in this paper is part of a project to apply the probabilistic method in topology. This methodology begins by defining a probability distribution on the objects of study. Parameters and invariants of interest then become random variables on this probability space. Tools of probability theory are then applied to investigate the distribution of these random variables. 

This general approach has yielded unexpected results in graph theory and many other areas. It has often provided existence proofs of objects with unexpected properties, such as expander graphs, or graphs of arbitrarily high girth and chromatic number. In many cases, such as for lower bounds for Ramsey numbers, finding matching explicit constructions remains open. 

In view of the great success of this paradigm in discrete mathematics, it is natural to consider its application to the study of random geometric objects. In this direction, the probabilistic method has played a major role in the theory of normed spaces for some time~\cite{milman1986asymptotic}. More recently there have been interesting attempts to study random simplicial complexes~\cite{linial2006homological}, random 3-manifolds~\cite{dunfield2006finite, maher2010random, kowalskicomplexity, lubotzky2014random, lutz2008combinatorial} and more. Our focus here is to bring this approach to random knots and links, and to associated knot and link invariants.

\subsection*{Model Dependence}

The random knots and links investigated in this paper are based on the Petaluma model, which determines a knot or link from a petal diagram and a permutation. It is important to consider to what extent our results are model dependent, and to investigate what might happen if we switch to a different model.  

To test the extent of model dependency of our statistics, we ran numerical studies on the distribution of the $c_2$ and $v_3$ invariants in a second random model, the grid model, discussed in Section~\ref{cassonsect}. Our numerical experiments yield a distribution for $c_2$ in the grid model. This distribution shares many features with the distribution obtained for the Petaluma model. We also derive some statistical measures in the star model, a related model introduced in Section~\ref{cassonsect}. The similarities between these three models are discussed in Section~\ref{cassonsect}. We note for example that in both the Petaluma and grid models the Casson invariant of a random knot appears to be positive more often than negative, at roughly a $3:1$ ratio.

It remains unclear how the choice of a random model determines the statistics of a knot and link invariant, and whether universality principles apply across a wide range of models.

\subsection*{Plan}

In Sections~\ref{linksect}, ~\ref{cassonsect}, and~\ref{j3sect}, we investigate invariants of order $1$, $2$, and $3$, respectively, in the random model: the linking number in Section~\ref{linksect}, the Casson invariant $c_2$ in Section~\ref{cassonsect}, and $v_3$ in Section~\ref{j3sect}. At the end of Section~\ref{cassonsect}, we present some calculations of $c_2$ moments and numerical results, and discuss their relations to other models.

\section{The Linking Number of a Random 2-Component Link}\label{linksect}

Before describing the linking number, we provide some general background on finite type invariants, an important family of knot and link invariants that includes the linking number.

\subsection*{Finite Type Invariants}

Finite type invariants were introduced by Vassiliev~\cite{vassiliev1988cohomology} and have been the subject of intensive study. These invariants are convenient to define via their extension to the more general context of singular knots and links, as we describe below.

A \emph{singular knot} is a smooth map of the circle $S^1$ to $\R^3$ with finitely many \emph{double} points of transversal self intersection. This is considered up to isotopy that preserves the double points. The more general definition of a \emph{singular link} is similar. A double point $p$ on a singular link $L$ can be {\em resolved} in two different ways, by locally pushing one of the strands in one direction or in the other. Locally, these resolutions come in two distinct well-defined forms: 
positive~\tikz[thick,->,baseline=-2]{\draw[black](0.3,-0.1) -- (0,0.2);\draw[white,-,line width=3](0.03,-0.1) -- (0.33,0.2);\draw[black](0.03,-0.1) -- (0.33,0.2);},
and negative~\tikz[thick,->,baseline=-2]{\draw[black](0.03,-0.1) -- (0.33,0.2);\draw[white,-,line width=3](0.3,-0.1) -- (0,0.2);\draw[black](0.3,-0.1) -- (0,0.2);}.
For such a double point $p$, we denote the resulting links by $L_p^+$ and $L_p^-$, each having one less singularity than~$L$.

Every link invariant is extended to singular links via the recursion
$ v(L) = v(L_p^+) - v(L_p^-)$.
The value of $v$ on a singular link with $m$ double points is thus a signed sum of its value on $2^m$ non-singular links. A knot (or link) invariant is of \emph{finite type}, or \emph{finite order}~$m$, if it vanishes on all singular knots (or links) with $m+1$ double points.

The first non-trivial finite type knot invariant is the Casson invariant $c_2(K)$, the second coefficient of the Conway polynomial~\cite{ball1981sequence, kauffman1981conway}. It is an invariant of order~2. See Section~\ref{cassonsect} for a constructive definition. The next independent invariant, of order~3, is determined by the third coefficient of the modified Jones polynomial, to be discussed in Section~\ref{j3sect}. In order~4 there are already three new invariants.
Vassiliev's conjecture states that finite type invariants distinguish knots. For more details see~\cite{chmutov2012introduction}.

\subsection*{The Linking Number}

A well-known invariant of two-component oriented links, the linking number $lk(L)$, is an invariant of order~$1$. The linking number counts the number of times that one component winds around the other, a number that is symmetric in the two components. It is defined here via link diagrams.
Recall that a \emph{link diagram} is a projection of the link to a plane,
that is one-to-one except for a finite number of double-points where two strands cross each other transversely.
At each such \emph{crossing point} it records which strand is upper and which is lower in the original link.
A crossing is \emph{positive} or \emph{negative} according to the orientation of the upper and lower strands as an ordered basis of the plane.

The linking number of a two-component link $L$ can be computed from its link diagram as follows.
Pick arbitrarily one of the components, and consider the crossings where it passes over the other one.
The linking number is the sum of the signs of these crossings.
Schematically, if the two components are colored gray and black,
$$ lk(L) \;\;=\;\; \#\;
\tikz[ultra thick,->,baseline=-4]{\draw[black](0.6,-0.3) -- (0,0.3);\draw[white,line width=5](0.05,-0.3) -- (0.65,0.3);\draw[lightgray](0.05,-0.3) -- (0.65,0.3);}
\;-\;\#\;
\tikz[ultra thick,->,baseline=-4]{\draw[black](0.05,-0.3) -- (0.65,0.3);\draw[white,line width=5](0.6,-0.3) -- (0,0.3);\draw[lightgray](0.6,-0.3) -- (0,0.3);}\;\;\;\;. $$
\begin{example}
$lk(\;
\tikz[ultra thick,baseline=-4]{
\draw[-,lightgray](0.6,0) .. controls +(-90:0.4) and +(-90:0.4) .. (0,0);
\draw[-,white,line width=4](1,0) .. controls +(-90:0.4) and +(-90:0.4) .. (0.4,0);
\draw[<-,black](1,0) .. controls +(-90:0.4) and +(-90:0.4) .. (0.4,0);
\draw[-,black](0.4,0) .. controls +(90:0.4) and +(90:0.4) .. (1,0);
\draw[-,white,line width=4](0,0) .. controls +(90:0.4) and +(90:0.4) .. (0.6,0);
\draw[<-,lightgray](0,0) .. controls +(90:0.4) and +(90:0.4) .. (0.6,0);}
\;)=1$, as there is one relevant crossing, and it is positive.
\end{example}

Consider a random link $L_{4n}$, obtained from a random permutation $\pi \in S_{4n}$ via a two-component petal diagram, as in Figure~\ref{fig_petal2}. The $4n$ strands pass at the center at $4n$ different heights, where strand $i$ passes above strand $j$ if $\pi(i)>\pi(j)$.

The strands are numbered according to the following rule.
Choose two base points on the large external loop of each component, say, at the two uppermost points of the diagram.
Start travelling from the upper right base point along the entire component, then take a similar trip throughout the other component. 
The strands are numbered from $1$ to $4n$ according to the order they are visited in this scan.
The heights occurring in the two components are thus given by
$$ X = \{\pi(i):1 \leq i \leq 2n\}, \;\;\;\;\;\;\;\; Y = \{\pi(j):2n+1 \leq j \leq 4n\}.$$

Note that the two components go alternatingly back and forth along two orthogonal axes.
We thus define a function that records the direction $S$ of the strand at each height $x$,
$$S(x) \;=\; (-1)^{\pi^{-1}(x)} \;\;\;\;\;\;\;\;\;\;\;\; 1 \leq x \leq 4n\;.$$
For example, if the strand at height $x \in X$ is directed SW to NE $(\nearrow)$, then $S(x)=1$.
Likewise, $S(y)=1$ for a strand at height $y \in Y$ with direction SE to NW $(\nwarrow)$.
In our model, $L_{4n}$, $S$, $X$, and $Y$ are all random variables, i.e. functions of $\pi$.

A slight perturbation of $L_{4n}$ near the center yields a link diagram with $\binom{4n}{2}$ simple crossings, of which $(2n)^2$ involve both components. If $x>y$
for strands $x \in X, y\in Y$, the sign of the corresponding crossing is $S(x)S(y)$.
We thus obtain the following formula for the linking number
\begin{equation}
lk(L_{4n}) \;=\; \ell(S,X,Y) \;:=\; \sum\limits_{x \in X}\;\sum\limits_{y \in Y} \;\begin{cases} S(x)S(y) & x>y \\ 0 & else \end{cases} 
\label{lkL4n}\tag{$\star$}
\end{equation}
This expression involves only the variables $S$, $X$, and $Y$ with no direct reference to $\pi$. Indeed the mapping $\pi\to (S,X,Y)$ is many-to-one. The linking number as given by this formula depends only on the partition of $[4n] = \{1,\dots,4n\}$ into the four parts of size $n$ each,
$$ X^+ = X \cap S^{-1}(1) \;\;\;\;\; X^- = X \cap S^{-1}(-1) \;\;\;\;\; Y^+ = Y \cap S^{-1}(1) \;\;\;\;\; Y^- = Y \cap S^{-1}(-1) \;.$$
The $(4n)!$ possible choices of $\pi$ split into $\binom{4n}{n,n,n,n}$ equal-sized subclasses, where the permutations in the same class yield the same linking number.

\subsection*{The Random Area Problem}

The sum in~\eqref{lkL4n} has an interesting interpretation.
We associate with the link $L_{4n}$ a random walk $\gamma=\gamma_{4n}:\{0,\dots,4n\} \rightarrow \Z^2$ on the grid.
The walk starts at the origin $\gamma(0) = (0,0)$, and makes $4n$ steps related to the $4n$ strands at the center of $L_{4n}$'s two-component petal diagram. Each step is determined by the direction of a strand in the diagram -- up, down, right, or left, where the strands are considered by their height at the crossing, from below to above.
$$ \gamma(t) = \gamma(t-1) + \begin{cases} (S(t),0) & \text{if } t \in X \\ (0,S(t)) & \text{if } t \in Y \end{cases} \;.$$
This walk is {\em balanced}. It takes $n$ steps in each of the four directions, and returns to the origin at $t=4n$. Note that $L_{4n}$ induces a uniform probability distribution on the $\binom{4n}{n,n,n,n}$ different walks. Connecting every two consecutive lattice points that $\gamma$ visits determines a polygonal path in the~$xy$-plane.

The \emph{algebraic area} $A(\gamma)$ enclosed by a closed grid walk, is the sum of $\gamma$'s winding numbers around all grid squares.
For example, if $\gamma$ is self-avoiding and oriented counterclockwise, it coincides with the regular notion of area enclosed by a curve.
In general, it can be computed by a ``discrete Green's theorem''.
A horizontal step at time $t$ contributes $\pm 1$ to each square in the column between the horizontal edge $(\gamma(t-1),\gamma(t))$ and the $x$-axis.
This yields
\begin{align*}
A(\gamma) \;=\; \oint_{\gamma} (-y)dx \;&=\; \sum\limits_{t=1}^{4n} (-\gamma_y(t)) \cdot (\gamma_x(t)-\gamma_x(t-1)) \\
\;&=\; -\sum\limits_{x \in X} \;\sum\limits_{\substack{y \in Y , x>y}} \;S(y) \; S(x) \;=\; -\ell(S,X,Y) \;.
\end{align*}
The distribution of the algebraic area is clearly symmetric around $0$, as seen by reflecting $\gamma$ around the $x$-axis.
Thus the minus sign can be ignored and the equality of distributions $lk(L_{4n})$ and $A(\gamma_{4n})$ follows for every $n$.

An imbalanced case of this problem appears in the literature.
Let $\Gamma_N$ be a uniformly chosen closed $N$-step random walk in $\Z^2$
with $l,r,u,d$ steps to the left, right, up and down, respectively. Clearly, $N=l+r+u+d$ is even, since $l=r$ and $u=d$. In the imbalanced case we do not demand that all four are equal.
The asymptotic distribution of $A(\Gamma_N)$ is known to be
$$ P\left[\alpha < \frac{A(\Gamma_N)}{N} < \beta \right] \;\xrightarrow{\;N\rightarrow\infty\;}\;
\int^{\beta}_{\alpha} \frac{\pi}{\cosh^2(2 \pi x)}dx \;=\; \frac{\tanh(2 \pi \beta) - \tanh(2 \pi \alpha)}{2} $$

This beautiful formula was established by three different approaches: by comparison to Brownian motion in the plane~\cite{levy1951wiener, khandekar1988distribution, duplantier1989areas}, through the Harper Equation~\cite{harper1955general, bellissard1997exact}, and via the method of moments~\cite{mingo1998distribution}.

\begin{figure}[htb]
\centering
\includegraphics[width=0.8\textwidth]{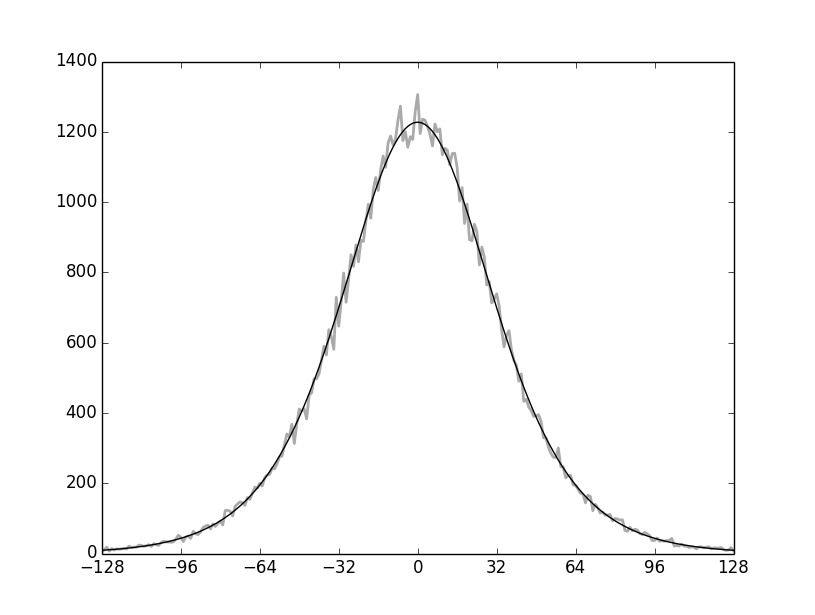}
\caption{Distribution of $lk(L_{4n})$ in $100000$ randomly sampled links with $n=64$,
compared to the predicted $1/\cosh^2$ limit distribution.}
\label{fig:L256}
\end{figure}

An imbalanced random $4n$-step closed walk, with high probability, takes $n\pm o(n)$ steps in each direction. Therefore it seems plausible that $A(\Gamma_{4n})/4n$ and $A(\gamma_{4n})/4n$ weakly converge to the same limit as $n \rightarrow \infty$, as stated in Theorem~\ref{linking}. In our proof, however, we do not take this route, but rather adopt the approach of Mingo and Nica~\cite{mingo1998distribution} and prove this from scratch. Some of our new ideas will be used later on when we investigate higher-order invariants.

An interesting consequence of Theorem~\ref{linking} is that a random link is almost surely non-trivial, since the probability that the linking number vanishes tends to zero.
In fact, numerical simulations (Figure~\ref{fig:L256}) suggest even a {\em local} limit distribution law for $lk(L_{4n})$.

We note also that Theorem~\ref{linking} can be extended to the case of imbalanced random links, with $2p$ and $2q$ loops in the two-component petal diagram. Replacing $n$ by $\sqrt{pq}$, one obtains the same limit distribution for $p,q \to \infty$. The case of $\Gamma_{N}$ then corresponds to the case where $p$ is also random with binomial distribution $B(N/2,1/2)$ and $p+q=N/2$.

As observed in~\cite{mashkevich2009area} and~\cite{mohammad2010enumeration}, for given $n$, the distribution of $L_{4n}$ can be computed in polynomial time.
To this end we define the functional $A(\gamma) = -\int_{\gamma} ydx$ on every planar walk. For $\gamma$ closed this is the signed area. Let $Z_{l,r,u,d}(A)$ be the number of walks $\gamma$ with $l,r,u,d$ steps in the four directions, and $A(\gamma)=A$.
Clearly $Z_{0000}(0) = 1$ and $Z_{0000}(A) = 0$ for $A \ne 0$.
The recurrence relation is
$$ Z_{l,r,u,d}(A) \;=\; Z_{l-1,r,u,d}(A+u-d)+Z_{l,r-1,u,d}(A-u+d)+Z_{l,r,u-1,d}(A)+Z_{l,r,u,d-1}(A) $$
In particular, we can compute $P[lk(L_{4n}) = k] \;=\; Z_{nnnn}(k) / \binom{4n}{n,n,n,n}$.

\subsection*{Proof of Theorems~\ref{moments1} and~\ref{linking}}

Recall from~\eqref{lkL4n} that
$$ lk(L_{4n}) \;=\; \ell(S,X,Y) \;=\; \sum\limits_{\substack{x \in X,y \in Y \\ x>y}} S(x)S(y) $$
where $x\in X$ and $y \in Y$ and $S$, $X$ and $Y$ are as above.
The main challenge is the computation of all moments of the random variable $\ell$. All odd moments of $\ell(S,X,Y)$ vanish due to the symmetry of $\ell$
with respect to flipping $S(x)$ for all $x \in X$. We turn to evaluate $E(\ell^k)$ for $k$ even.

The coordinate-wise order among vectors $\textbf{x},\textbf{y} \in \Z^k$ is denoted $\textbf{x} > \textbf{y}$, i.e., $x_i > y_i$ for all $1 \leq i \leq k$.
For fixed $S$, $X$, and $Y$, the $k$th power of the linking number is given by
\begin{align*}
\ell(S,X,Y)^k \;&=\; \sum\limits_{x_1 > y_1} S(x_1)S(y_1) \sum\limits_{x_2 > y_2} S(x_2)S(y_2) \cdots \sum\limits_{x_k > y_k} S(x_k)S(y_k) \\
\;&=\; \sum\limits_{\textbf{x} \in X^k} \; \sum\limits_{\textbf{y} \in Y^k} \;\; \mathcal{I}\left[\textbf{x} > \textbf{y}\right] \; \prod\limits_{i=1}^{k} S(x_i)S(y_i)
\end{align*}
where $\mathcal{I}\left[condition\right]:=1$ if the condition holds, and else $0$.

We split the terms in this sum according to collisions in $\textbf{x}$ and in $\textbf{y}$, i.e. indices $i \neq j$ for which $x_i=x_j$ or $y_i=y_j$.
Every $\textbf{x} \in X^k$ induces a partition of $[k]$ denoted $p_{\textbf{x}}$. Each part has the form $\{i : x_i = t \}$ for some $t \in X$ assuming this set is nonempty. So we first sum over pairs of partitions, $\xi \vdash [k]$ and $\eta \vdash [k]$,
and then over vectors $\textbf{x}$ and $\textbf{y}$ such that $p_{\textbf{x}} = \xi$ and $p_{\textbf{y}} = \eta$.

Let $\xi = \{\xi_1, \xi_2 ,\dots\} = p_{\textbf{x}}$ for some $\textbf{x} \in X^k$. Recall that as $i\in \xi_j$ varies, all $x_i$ are equal. Therefore $S$ induces a sign on each part of $p_{\textbf{x}}$ via $s_{\textbf{x}}(\xi_j) = S(x_i)\in\{-1,1\}$. We refine the outer summation by fixing sign functions, $\varepsilon:\xi \rightarrow \pm 1$ and $\delta:\eta \rightarrow \pm 1$, and summing separately over vectors $\textbf{x}$ and $\textbf{y}$ with $s_{\textbf{x}} = \varepsilon$ and $s_{\textbf{y}} = \delta$.

Let $\textbf{x} \in X^k$ have $p_{\textbf{x}} = \xi$ and $s_{\textbf{x}} = \varepsilon$. The term $\prod_i S(x_i)$ that $\textbf{x}$ contributes to the sum is
expressible in terms of $\xi$ and $\varepsilon$ as
$$ s(\xi,\varepsilon) \;:=\; \prod\limits_{\xi_j \in \xi} \varepsilon(\xi_j)^{|\xi_j|} \;\;.$$
We can now rewrite
$$ \ell(S,X,Y)^k \;=\; \sum\limits_{\xi\; \vdash [k]} \;\; \sum\limits_{\eta\; \vdash [k]} \;
\; \sum\limits_{\varepsilon:\xi \rightarrow \pm 1} \; \sum\limits_{\delta:\eta \rightarrow \pm 1}
\#\left\{\textbf{x} \in X^k , \textbf{y} \in Y^k \;\left|\;
\substack{p_{\textbf{x}}=\xi,\;p_{\textbf{y}}=\eta, \\ s_{\textbf{x}}=\varepsilon,\;s_{\textbf{y}}=\delta,}\right. \textbf{x} > \textbf{y} \right\}
\cdot s(\xi,\varepsilon) s(\eta,\delta) $$

Given $\xi,\eta,\varepsilon,\delta$, how many pairs of vectors $\textbf{x},\textbf{y}$ satisfy the conditions in the curly brackets?
Recall that the range $\{1,\dots,4n\}$ consists of four sets of size $n$:
$$ X^+ := X \cap S^{-1}(1) \;\;\;\;\; X^- := X \cap S^{-1}(-1) \;\;\;\;\; Y^+ := Y \cap S^{-1}(1) \;\;\;\;\; Y^- := Y \cap S^{-1}(-1) $$
The number of parts in a partition $\xi$ is denoted $|\xi|$.
Note that there are $|\xi|+|\eta|$ distinct elements that appear in $\textbf{x}$ and $\textbf{y}$.
The number of such elements from each of the above four sets is respectively
$$ x^+ \;:=\; |\varepsilon^{-1}(1)| \;\;\;\;\; x^- \;:=\; |\varepsilon^{-1}(-1)| \;\;\;\;\; y^+ \;:=\; |\delta^{-1}(1)| \;\;\;\;\; y^- \;:=\; |\delta^{-1}(-1)| $$
so that $x^++x^-=|\xi|$ and $y^++y^-=|\eta|$.
One can choose $x^+$ distinct values from $X^+$ in $(n)_{x+}$ ways, where $(n)_m = n(n-1)(n-2)\cdots(n-m+1)$.
There are thus $(n)_{x+}(n)_{x-}(n)_{y+}(n)_{y-}$ pairs of vectors $\textbf{x}$ and $\textbf{y}$
compatible with the partitions $\xi$ and $\eta$, and with the sign functions $\varepsilon$ and $\delta$.

We now turn to account for the condition $\textbf{x} > \textbf{y}$. To this end
it is convenient to think of $\textbf{x}$ and $\textbf{y}$ as uniformly picked at random from all the above pairs.
The desired number is then $(n)_{x+}(n)_{x-}(n)_{y+}(n)_{y-}$ times the probability of the event $\textbf{x} > \textbf{y}$.
This probability is denoted $P_{\xi\eta\varepsilon\delta}\left[\textbf{x} > \textbf{y}\;|\;S,X,Y\right]$, since here $S$, $X$ and $Y$ are fixed, and $\textbf{x},\textbf{y}$ that are compatible with $\left(\xi,\eta,\varepsilon,\delta\right)$ are sampled at random.
Plugging it into the sum, we obtain
$$ \ell(S,X,Y)^k \;=\; \sum\limits_{\xi,\eta\; \vdash [k]} \; \; \sum\limits_{\varepsilon:\xi \rightarrow \pm 1} \; \sum\limits_{\delta:\eta \rightarrow \pm 1}
P_{\xi\eta\varepsilon\delta}\left[\textbf{x} > \textbf{y}\;|\;S,X,Y\right]\cdot(n)_{x+}(n)_{x-}(n)_{y+}(n)_{y-} \cdot s(\xi,\varepsilon) \; s(\eta,\delta).$$

To compute the $k$th moment of $\ell$, we average this over all $S$, $X$ and $Y$.
Note that $S$, $X$ and $Y$ appear only in the factor $P_{\xi\eta\varepsilon\delta}\left[\textbf{x} > \textbf{y}\;|\;S,X,Y\right]$.
Therefore, by linearity of the expectation, it is sufficient to compute it for each such factor.

A~priori, this would require a summation of $\binom{4n}{n,n,n,n}$ different probabilities.
However, we can view $\textbf{x}$ and $\textbf{y}$ as random variables in a combined probability space,
where one first uniformly picks $S$, $X$, and $Y$, and then $\textbf{x} \in X^k$ and $\textbf{y} \in Y^k$ that are compatible with $\left(\xi,\eta,\varepsilon,\delta\right)$.
By the law of total probability,
$$ E_{S,X,Y} \left[P_{\xi\eta\varepsilon\delta}\left[\textbf{x} > \textbf{y}\;|\;S,X,Y\right]\right] \;=\;
P_{\xi\eta\varepsilon\delta}\left[\textbf{x} > \textbf{y} \right] $$
where $P_{\xi\eta\varepsilon\delta}\left[\textbf{x} > \textbf{y} \right]$ is the probability in the combined probability space.

Note that the event $(\textbf{x} > \textbf{y}) = \bigcap_{i=1}^k \;(x_i > y_i)$ only depends on the order relation among the $|\xi|+|\eta|$ values in the range $\{1,\dots,4n\}$,
that occur as entries in $\textbf{x}$ and $\textbf{y}$.
These can be encoded by a one-to-one function
$$ \sigma \;:\; \xi \sqcup \eta \;\rightarrow\; \{1,\dots,|\xi|+|\eta|\} $$
such that the correspondence that sends $x_i$ to $\sigma(\xi_j)$ where $i \in \xi_j$, and $y_i$ to $\sigma(\eta_j)$ where $i \in \eta_j$, is order-preserving.
For example, the condition $x_1 > y_1$ takes the form $\sigma(\xi_i) > \sigma(\eta_j)$ where $i,j$ are such that $1 \in \xi_i \cap \eta_j$.

Since we are sampling uniformly, each of the $(|\xi|+|\eta|)!$ choices for $\sigma$ is equally likely in the combined probability space,
and the probability $P_{\xi\eta\varepsilon\delta}\left[\textbf{x} > \textbf{y} \right]$ is
$$ P(\xi,\eta) \;:=\; \frac{\# \left\{\sigma:(\xi \sqcup \eta) \leftrightarrow \{1,\dots,|\xi|+|\eta|\}
\;\Big|\; \text{If}\;\; \xi_i \cap \eta_j \neq \varnothing \;\;\text{then}\;\; \sigma(\xi_i) > \sigma(\eta_j) \right\}}{(|\xi|+|\eta|)!} $$
In conclusion, the $k$th moment of the linking number is
$$ E\left[lk(L_{4n})^k\right] \;=\; \sum\limits_{\xi,\eta\; \vdash [k]} \; P(\xi,\eta)
\sum\limits_{\varepsilon:\xi \rightarrow \pm 1} (n)_{x+}(n)_{x-} \cdot s(\xi,\varepsilon)
\sum\limits_{\delta:\eta \rightarrow \pm 1} (n)_{y+}(n)_{y-} \cdot s(\eta,\delta) $$
Our calculations thus far show that the $k$th moment is a polynomial in $n$ of degree at most $2k$, since $\deg_n [(n)_{x+}(n)_{x-}] = |\xi| \leq k$ and $\deg_n [(n)_{y+}(n)_{y-}] = |\eta| \leq k$. The following lemma reduces the degree down to $k$.

\begin{lemma}\label{lnlemma}
Let $\xi \; \vdash [k]$ where $k$ is even,
and for $\varepsilon:\xi \rightarrow \pm 1$ denote $x^{\pm} = |\varepsilon^{-1}(\pm 1)|$ and $s(\xi,\varepsilon) = \prod\limits_{\xi_j \in \xi} \varepsilon(\xi_j)^{|\xi_j|}$.
Then
$$ \left(\frac{-1}{2n}\right)^{k/2} \sum\limits_{\varepsilon:\xi \rightarrow \pm 1} (n)_{x+}(n)_{x-} \cdot s(\xi,\varepsilon)
\;\;\xrightarrow{n\rightarrow\infty}\;\; \begin{cases} 0 & \max_j|\xi_j| \geq 3 \\
(-1)^{|\xi|}(u(\xi)-1)!! & \text{else} \end{cases}$$
where $u(\xi) := \#\{j : |\xi_j|=1\}$ and as usual $(2r-1)!! = (2r-1)(2r-3)\cdots 1 = (2r)!/2^rr!$.
\end{lemma}

\begin{proof}
We use two standard combinatorial identities about sums.
\begin{enumerate}
\item
For non-negative integers $n,a,b$,
$$(n)_a(n)_b = \sum_{r=0}^{a+b} \frac{(a)_r(b)_r}{r!} \;(n)_{a+b-r} \;\;\;.$$
Both sides count pairs of words in the alphabet $[n]$ the first having $a$ distinct letters and the second having $b$ distinct letters. The right hand side splits the summation according to the number $r$ of letters that appear in both words.
Since $(a)_r=0$ for $r>a$, the sum is in fact only up to $r=\min(a,b)$.
\item
For $I\subseteq \{1,\dots,t\}$ we define $\chi(I):\{\pm 1\}^t \to \R$ by $\chi(I)(\varepsilon_1,\dots,\varepsilon_t) = \prod_{i \in I}\varepsilon_i$.
The functions $\left\{\left.\chi(I)\;\right|\; I \subseteq \{1,\dots,t\} \right\}$ constitute an orthogonal basis of the linear space $\mathcal{W}_t$ of functions $\{\pm 1\}^t \rightarrow \mathbb{R}$. Namely,
$$ \left\langle \chi(I),\chi(J)\right\rangle \;=\; \sum\limits_{{\varepsilon} \in \{\pm 1\}^t} \chi(I)({\varepsilon})\chi(J)({\varepsilon}) \;=\; 2^t \delta_{IJ} \;\;.$$
This is, in fact, the Fourier basis of $\mathcal{W}_t$ with respect to the group $\Z_2^t$. To verify these relations note that $\sum_{\varepsilon}\chi(I)(\varepsilon) = 0$ for $I \neq \varnothing$, and $\chi(I)\chi(J) = \chi(I \triangle J)$.
\end{enumerate}
We now turn to prove the statement of the lemma. By the above identity,
$$ \sum\limits_{\varepsilon:\xi\to\pm 1} (n)_{x+}(n)_{x-} \cdot s(\xi,\varepsilon)
\;=\; \sum_{r=0}^{|\xi|} \;\frac{(n)_{|\xi|-r}}{r!} \sum\limits_{\varepsilon:\xi\to\pm 1} \left(x^+\right)_r\left(x^-\right)_r \cdot s(\xi,\varepsilon).$$
Let $t = |\xi|$, and $J = \{j : |\xi_j| \text{ is odd}\} \subseteq \{1,\dots,t\}$.
Writing $x^{\pm}$ in terms of $\varepsilon$, the above expression takes the form
$$\sum_{r=0}^{t} \;\frac{(n)_{t-r}}{r!} \; \sum\limits_{\varepsilon:\xi \rightarrow \pm 1}
\left(\frac{t}{2} + \frac12\sum_{j=1}^t\varepsilon(\xi_j)\right)_{\mathlarger{r}}\left(\frac{t}{2} - \frac12\sum_{j=1}^t\varepsilon(\xi_j)\right)_{\mathlarger{r}}
\cdot \;\prod\limits_{j \in J} \varepsilon(\xi_j) $$
The sum over $\varepsilon$ can be viewed as the inner product $\left\langle \cdot, \chi(J) \right\rangle$ in $\mathcal{W}_t$. Expanding the product of the $2r$ factors in $\left(x^+\right)_r\left(x^-\right)_r$, we obtain a linear combination of $\{\chi(I)\}_{|I| \leq 2r}$.

Note that $|J|$ is always even since $k$ is. Note also that $|J| \geq 2t-k$ with equality if and only if all parts of $\xi$ are of size $1$ or $2$. The analysis splits into several cases:
\begin{itemize}
\item
If $r<|J|/2$ then there are no products of $|J|$ different $\varepsilon(\xi_j)$'s in the expansion,
so that it is orthogonal to $\chi(J)$, and such $r$ can be ignored.
\item
If {\bf $r>|J|/2$} then $k/2>t-r$ and so $(n)_{t-r} = o(n^{k/2})$. Such $r$
also contribute zero to the limit in the lemma.
\item
If $r=|J|/2$ and $\max_j|\xi_j| \geq 3$, so that $|J| > 2t-k$ strictly, then the term is similarly $o(n^{k/2})$.
This implies the lemma in the first case, where the limit is $0$.
\item
In the remaining case $|\xi_j| \leq 2$ for all $j$, and $r=|J|/2=t-k/2$. By orthogonality we only have to count the occurrences of $\chi(J)$ in the expansion on the left.
Being a product of $2r$ distinct $\varepsilon(\xi_j)$'s, it appears there exactly $(2r)!$ times, with a coefficient of $(-1)^r/2^{2r}$.
Since $\left\langle \chi(J),\chi(J)\right\rangle=2^t$ this term contributes
$$ (n)_{t-r}/r! \cdot (2r)! \cdot (-1)^r/2^{2r} \cdot 2^t \;=\; (-1)^{t-k/2} \cdot (2t-k-1)!! \cdot 2^{k/2} \cdot \left(n^{k/2} + O(n^{k/2-1})\right) $$
which proves the second case of the lemma. \qedhere
\end{itemize}
\end{proof}

Theorem~\ref{moments1} follows immediately from Lemma~\ref{lnlemma}. We use the lemma to establish also the limit distribution, and consider
$$ \Lambda_k \;:=\; \lim_{n \rightarrow \infty} E\left[\left(\frac{lk(L_{4n})}{2n}\right)^k\right] \;=\;
\sum\limits_{\substack{\xi,\eta \;\vdash [k] \\ |\xi_j|,|\eta_i| \leq 2}} \; P(\xi,\eta) \cdot (-1)^{|\xi|+|\eta|} (u(\xi)-1)!! (u(\eta)-1)!! $$

Thus $\Gamma_{4n}$ and $\gamma_{4n}$ have the same limit behavior. Indeed the limit of the moment $\Lambda_k$ is precisely as in Proposition 4.3 on page 74 of~\cite{mingo1998distribution}, and the rest of the proof is a modification of their argument in pages 75--85. This involves further simplification of $\Lambda_k$, finding the moment generating function, and application of the method of moments to obtain the limit distribution. It is still useful to review these steps here, since similar arguments are relevant in our analysis of other  invariants below.

For two partitions $\xi, \xi' \;\vdash [k]$, we write $\xi \prec \xi'$ if $\xi$ is a \emph{refinement} of $\xi'$, i.e., every part of $\xi$ is contained in a part of $\xi'$.
If $\xi$ is a partition of $[k]$ with parts of size $1$ or $2$,
then there are $(u(\xi)-1)!!$ partitions $\xi'$ all of whose parts have size $2$ with $\xi \prec \xi'$. They correspond to all perfect matchings between the singletons in $\xi$.
Define
$$ \psi(\xi',\eta') \;:=\; \sum\limits_{\xi \prec \xi', \eta \prec \eta'} (-1)^{|\xi|+|\eta|} P(\xi,\eta) $$
and rewrite the sum as
$$ \Lambda_k \;=\;
\sum\limits_{\substack{\xi,\eta \;\vdash [k] \\ |\xi_j|,|\eta_i| \leq 2}} \; P(\xi,\eta) \cdot (-1)^{|\xi|+|\eta|}
\cdot \#\left\{\xi',\eta' \;\left|\; \substack{ \xi \prec \xi', \eta \prec \eta' \\ |\xi'_j|,|\eta'_i| = 2} \right.\right\}
\;=\; \sum\limits_{\substack{\xi',\eta' \;\vdash [k] \\ |\xi'_j|,|\eta'_i| = 2}} \psi(\xi',\eta') \;.$$

Two partitions $\xi',\eta' \vdash [k]$ define a bipartite \emph{intersection graph} $G(\xi',\eta')$ with sides $\xi'$ and $\eta'$, where the number of edges between $\xi'_j$ and $\eta'_i$ is $|\xi'_j \cap \eta'_i|$, the size of their intersection.
Note that the size of $\xi'_i$ is the degree of the corresponding vertex in $G(\xi',\eta')$.
In our case where $|\xi'_j|,|\eta'_i| = 2$, the graph is $2$-regular, and so it is a disjoint union of $m$ cycles of even lengths $l_1,l_2,\dots,l_m$.
By definition of $P(\xi,\eta)$, it is easy to see that $\psi(\xi',\eta')$ depends only on the intersection graph, and not on the vertex labels.

Here Mingo and Nica derive the exact value of $\psi(\xi',\eta')$, relying on classical work of D. Andr\'{e} on alternating permutations.
The same formula will be proved in a more general context in Lemma~\ref{cycle} below.
In this case, it gives
$$ \psi(\xi',\eta') \;=\; \beta_{l_1}\beta_{l_2}\cdots\beta_{l_m}\;\;.$$
where for $l$ even $\beta_l = (-1)^{l/2+1}B_{l}/l!$ and $B_{l}$ is the $l$th Bernoulli number~\cite[p. 1040]{jeffrey2007table}. For odd $l$ we set $\beta_l=0$.

In conclusion, $\Lambda_k$ is a sum of such products,
going over all pairs of $[k]$-partitions whose intersection graph is a disjoint union of even cycles.

We separate the sum according to the number $m$ of cycles in the intersection graph.
In addition to summing $\psi(\xi',\eta')$ over pairs of partitions, it is convenient to sum over all $m!$ ways to number the cycles in their intersection graph with $\{1,\dots,m\}$, and then divide by $m!$.

Consider a sequence of $m$ positive even integers $l_1,l_2,\dots,l_m$ such that $\sum_il_i=k$.
Let $G$ be a bipartite graph consisting of $m$ cycles, where cycle number $i$ has length $l_i$.
How many times does $G$ appear in the summation as an intersection graph with numbered cycles?
Every assignment of $[k]$ to the edges of $G$ yields a pair of suitable partitions corresponding to the two sides of the graph as follows.
The parts corresponding to each side of the graph consist of two labels each,
which are assigned to the two edges incident to a vertex on the given side.
A priori, there are $k!$ such assignments, but this should be corrected for symmetries.
There are $l_i$ ways to permute the assignments of the $i$th cycle without changing the pairs it contributes to each partition.
In conclusion, $G$ appears in the sum $k!/l_1l_2 \cdots l_m$ times.

Thus we express $\Lambda_k$ by summation over all ordered decompositions of $k$ into $m$ even terms:
$$ \Lambda_k \;=\; \sum\limits_{m}\;\sum\limits_{\substack{l_1+l_2+\dots+l_m=k\\l_i>0, \text{ even}}} \frac{k!\cdot\beta_{l_1}\beta_{l_2}\cdots\beta_{l_m}}{m!\cdot l_1l_2 \cdots l_m} \;\;.$$
The corresponding exponential generating function has a nice form
$$ {\mathcal L}(z) \;:=\; \sum_{k=0}^{\infty} \frac{\Lambda_k}{k!} z^k
\;=\; \sum\limits_{m=0}^{\infty}\frac{1}{m!}\;\prod\limits_{j=1}^m \;\sum\limits_{l_j=1}^{\infty} \frac{\beta_{l_j}}{l_j}z^{l_j}
\;=\; \exp \left(\sum_{l=1}^{\infty} \frac{\beta_l}{l} z^l \right) \;\;.$$
By~\cite[p. 42]{jeffrey2007table},
$$ z\frac{d}{dz}\log {\mathcal L}(z) \;=\; \sum\limits_{l=1}^{\infty}\beta_lz^l \;=\;
\sum\limits_{m=1}^{\infty} \frac{(-1)^{m+1} B_{2m}}{(2m)!} z^{2m} \;=\; 1 - \frac{z}{2} \cdot \cot\left(\frac{z}{2}\right) \;\;,$$
which yields
$$ {\mathcal L}(z) \;=\; \frac{z/2}{\sin(z/2)} \;,$$
The exponential moment generating function ${\mathcal L}(z)$, of the limiting moments $\{\Lambda_k\}_{k=0}^{\infty}$, is analytic at~$0$. By the method of moments~\cite[Chapter 8]{mingo1998distribution,breiman1992probability}, this means that the limit distribution of $lk(L_{4n})/2n$ is uniquely determined by these moments, and obtained by the inverse Fourier transform, ${\cal F}^{-1}\left[{\mathcal L}(iz)\right](t) = \pi/2\cosh^2(\pi t)$~\cite[p. 1120]{jeffrey2007table}.
Theorem~\ref{linking} follows.\qed

\section{The Casson Invariant of a Random Knot}\label{cassonsect}

We now consider random knots. The Casson invariant $c_2$ is the second coefficient of the Conway polynomial~(\cite{kauffman1987knots}, Chapter 3). On the unknot it vanishes: $c_2(\bigcirc) = 0$. The extension of $c_2$ to a singular knot with one double point turns out to be the linking number of the two-component link, obtained by the \emph{smoothing} operation that replaces \rlap{$\nearrow$}\rlap{$\searrow$} $\;\;\;$ with $\mathlarger{\mathlarger{\mathlarger{\mathlarger{ \substack{\curvearrowbotright \\[-3pt] \curvearrowright \\[-5pt]}}}}}$ at the singular point. This is enough to determine $c_2$ for all knots. Indeed, every knot can be unknotted passing only through a finite number of singular knots with one double point. The Casson invariant is then the sum of the corresponding linking numbers, with appropriate signs. The outcome does not depend on the specific choices of unknotting steps, see, e.g.,~\cite{ball1981sequence}.

This definition leads to a formula for the Casson invariant of a knot given any of its diagrams~\cite{polyak1994gauss}. Choose a base point, and travel along the knot. A crossing is \emph{descending} if its upper strand is visited before its lower one, and \emph{ascending} otherwise.

\begin{lemma}\label{c2formula}
Given a knot diagram of $K$ with a base point,
$$ c_2(K) \;=\; \sum\limits_{p,q} \; \text{sign}(p) \cdot \text{sign}(q) \;. $$
where the sum is over pairs of crossings $(p,q)$ that are encountered traveling along $K$ from the base point in the order $p,q,p,q$ with $p$ ascending and $q$ descending.
\end{lemma}

\begin{proof}
To prove this formula, flip each ascending crossing point $p$ at its first visit along the travel. This process terminates when we return to the base point. The resulting knot diagram is clearly always descending and represents the unknot. How does $c_2$ change as we flip $p$? According to the definition of $c_2$ on singular knots, the change is $\pm$ the linking number of the diagram smoothed at $p$, at the moment of the $p$-flip.

As explained in Section~\ref{linksect}, this linking number is a sum of signs over certain crossing points. One component of the smoothed link contains the base point and the other does not. We consider all crossing points $q$ where the latter passes above the former. Since at the moment of the $p$-flip all crossings prior to the first visit at $p$ are already descending, such $q$ must have a lower strand between the second visit at $p$ and the return to the base point. But such crossings $q$ haven't been visited yet. They are characterized as being descending in the original diagram, as stated in the formula.
\end{proof}

\begin{rmk*}
This formula is a member of the large class of Gauss diagram formulas, which involve the numbers of certain configurations of crossings in knot diagrams. See Section~\ref{j3sect} for the general definition.
\end{rmk*}

\begin{figure}[htb]
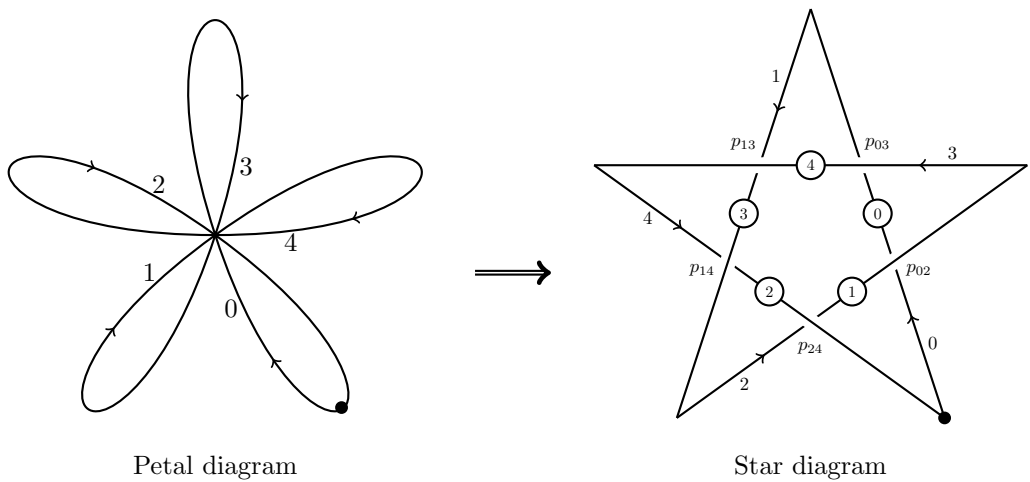

\begin{center}
\begin{tabular}{ccc}
\tikz[thick]{
\foreach \angle in {0, 72, ..., 288} \draw[decoration={markings,mark=at position 0.7 with {\arrow{>}}},postaction={decorate}] (0,0) .. controls +(\angle+36:4) and +(\angle:4) .. (0,0);
\filldraw (0,0)+(306:2.83) circle (2pt);
\foreach \angle/\num in {0/4, 72/3, 144/2, 216/1, 288/0} \draw (0,0)+(\angle-6:1) node[scale=1] {$\num$};
\pgfresetboundingbox \clip (-3,-2.75) rectangle (3,3);} &
\tikz[thick]{
\draw[->,double,thick] (-0.5,-0.5) -- (0.5,-0.5);
\pgfresetboundingbox \clip (-0.5,-2.75) rectangle (0.5,3);} &
\tikz[thick]{
\foreach \x in {0,1,...,4} {
\draw[>=triangle 90 cap,white,<->,line width = 6,shorten >=3, shorten <=3] (\x*288-54:3) -- (\x*288+90:3);
\draw[black,decoration={markings,mark=at position 0.25 with {\arrow{>}}},postaction={decorate}] (\x*288-54:3) -- node[fill=white,draw,circle,scale=0.6]{\x} (\x*288+90:3); }
\foreach \x in {0,1,...,4} {
\draw (\x*144-54:3)+(\x*144+99:1) node[auto,scale=0.7] {\x};}
\foreach \p/\x in {24/0,03/1,14/2,02/3,13/4} \node[scale=0.7] at (\x*144-90:1.5){$p_{\p}$};
\filldraw (306:3) circle (2pt);
\pgfresetboundingbox \clip (-3,-2.75) rectangle (3,3);
} \\
Petal diagram & & Star diagram
\end{tabular}
\end{center}
\caption{Petal to star diagram for $\pi=(0,3,1,4,2)$. The circled label attached to a segment gives its height, while the uncircled ones denote the segments' numbering.}
\label{fig_star5}
\end{figure}

In order to apply Lemma~\ref{c2formula} to our random knots, we turn a petal diagram into an ordinary knot diagram, with only simple crossings. Following Adams et al.~\cite{adams2013bounds}, we do so by straightening segments between petal tips, thus obtaining an equivalent \emph{star diagram}, as demonstrated in Figure~\ref{fig_star5}. Note that the horizontal segments are joined above the vertices of the star by vertical ones, that project to a point in the diagram. We always choose the orientation of the petal and corresponding star diagram as in Figure~\ref{fig_star5}. Starting at the same base point as in the petal diagram, we identify the segments with the integers $\{0,\dots,2n\}$. Example~\ref{trefoil} goes on with the computation of the Casson invariant of this knot.

\begin{example}\label{trefoil}
We use the diagram in Figure~\ref{fig_star5} to compute the Casson invariant for $\pi = (0,3,1,4,2)$. This means that segments $0,1,2,3,4$ are at heights $0,3,1,4,2$ respectively. Denote by $p_{\alpha\beta}$ the crossing of segment $\alpha$ with segment $\beta$. Note that $p_{02}$, $p_{03}$, $p_{13}$ and $p_{24}$ are ascending and $p_{14}$ is descending. Traveling around the knot tells us that the sum in Lemma~\ref{c2formula} is over the pairs $(p_{02},p_{14})$, $(p_{03},p_{14})$, $(p_{13},p_{14})$. Since $\text{sign}(p_{02}) = \text{sign}(p_{13}) = \text{sign}(p_{14}) = +1$ and $\text{sign}(p_{03}) = -1$, the Casson invariant is
$$ c_2(K_5(\pi)) \;=\; (+1)(+1) + (-1)(+1) + (+1)(+1) \;=\; 1 \;.$$
This diagram represents the \emph{positive trefoil} knot.
\end{example}

Let $\pi \in S_{2n+1}$ and $K=K_{2n+1}(\pi)$. We derive a general expression for $c_2(K)$ using $K$'s star diagram. We first describe the crossings along $K$ and the pairs relevant to Lemma~\ref{c2formula}, and then sum their contributions.

\begin{figure}[htb]
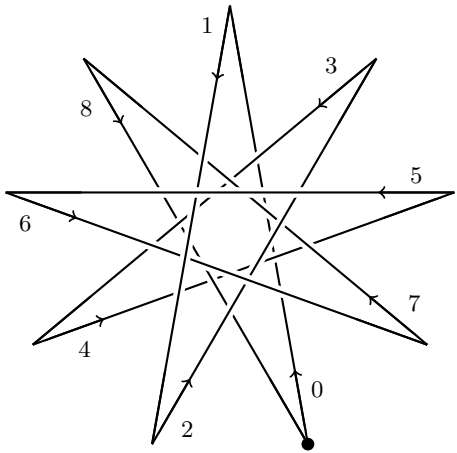

\begin{center}
\tikz[thick]{
\foreach \x in {0,1,...,8} {\draw[white,-,line width = 4] (\x*280-70:3) -- (\x*280+90:3); \draw[black] (\x*280-70:3) -- (\x*280+90:3); }
\foreach \x in {0,1,...,8} { \draw[->] (\x*160-70:3) -- node[auto,swap,pos=0.5]{\small\x} +(\x*160+100:1); \draw (\x*280-70:3) -- +(\x*280+120:1); }
\filldraw (290:3) circle (2pt);
\pgfresetboundingbox \clip (-3,-2.5) rectangle (3,3);}
\end{center}
\caption{Star diagram of some $9$-petal knot.}
\label{fig_star9}
\end{figure}

The underlying curve of the $2n+1$ star diagram is fixed for each $n$, while the permutation $\pi$ determines only the overcrossing/undercrossing information. In particular, each segment meets all non-adjacent segments in a fixed order. For example, in Figure~\ref{fig_star9} where $n=4$, segment $0$ crosses segments $6,4,2,7,5,3$ in this order, regardless of $\pi$. The general ordering is similarly given in the following proposition.

\begin{prop}\label{order}
A segment $\alpha \in \{0,\dots,2n\}$ in the star diagram meets segments
$$\alpha-3,\; \alpha-5,\; \dots,\; \alpha+4,\; \alpha+2,$$
and then
$$\alpha-2,\; \alpha-4,\; \dots,\; \alpha+5,\; \alpha+3.$$
In these two sequences, the segment number decreases by $2$ modulo $(2n+1)$.
\end{prop}

We consider pairs of crossings $p$ and $q$ that contribute to $c_2$. As we traverse the curve from the base point, the crossings occur successively on segments $\alpha \leq \beta \leq \gamma \leq \delta$. By the $p,q,p,q$ condition in Lemma~\ref{c2formula}, the two crossings are $p=p_{\alpha\gamma}$ and $q=p_{\beta\delta}$.

If the segments are distinct, so that $\alpha < \beta < \gamma < \delta$, then this pair of crossings participates in the $c_2$ formula whenever $p_{\alpha\gamma}$ is ascending and $p_{\beta\delta}$ is descending. There are also potentially relevant pairs of crossings that lie on a common segment, e.g. where $\alpha = \beta < \gamma < \delta$. Here $p_{\alpha\gamma}$ must precede $p_{\beta\delta}$ in the joint segment $\alpha = \beta$ in order to satisfy the $p,q,p,q$ condition, and we turn to characterize these pairs.

\newlength{\myparindent}\setlength{\myparindent}{\parindent}
\begin{center}
\begin{tabular}{@{}m{0.5\textwidth}@{}@{}m{0.5\textwidth}@{}}
\hspace{\myparindent} 
By Proposition~\ref{order}, the crossing points $p_{\alpha x}$, $x \neq \alpha \pm 1$, lie along the segment $\alpha$ according to the following order.
\begin{enumerate}
\item $x<\alpha$, $\;x \not \equiv \alpha \mod 2$.
\item $x>\alpha$, $\;x \equiv \alpha \mod 2$.
\item $x<\alpha$, $\;x \equiv \alpha \mod 2$.
\item $x>\alpha$, $\;x \not \equiv \alpha \mod 2$.
\end{enumerate} 
Inside each part, the segment ordering is decreasing. This is illustrated on the right for the $21$-star diagram with $\alpha=12$. & 
\begin{center}
\raisebox{-0.5\totalheight}{\tikz{\foreach \x/\y in {0/9,1/7,2/5,3/3,4/1,5/20,6/18,7/16,8/14,9/10,10/8,11/6,12/4,13/2,14/0,15/19,16/17,17/15} \draw (0.3*\x,0.2*\y) circle[radius=0.2] node[scale=0.8] {$\y$}; }}
\end{center} \\
\end{tabular}
\end{center}

Therefore, if $\alpha = \beta < \gamma < \delta$ and $p_{\alpha\gamma}$ precedes $p_{\alpha\delta}$ then necessarily $\alpha$ meets $\gamma$ in the second part and $\delta$ in the fourth. Hence $\gamma \equiv \alpha$ and $\delta \not\equiv \alpha\mod 2$. One may likewise examine the two cases $\beta = \gamma$ and $\gamma = \delta$, and conclude similar parity conditions. We cannot have more than one equality, since a segment doesn't cross itself, and two segments cross at most once. Let
$$
\mathcal{Q} \;:=\; \left\{ (\alpha, \beta, \gamma, \delta) \in \{0,\dots,2n\}^4 \;\left|\;
\begin{array}{l}
\alpha < \beta < \gamma < \delta \\
\alpha = \beta < \gamma < \delta \;\;\;\text{and}\;\;\;
\alpha \equiv \beta \equiv \gamma \not\equiv \delta \mod 2 \\
\alpha < \beta = \gamma < \delta \;\;\;\text{and}\;\;\;
\alpha \equiv \beta \equiv \gamma \equiv \delta \mod 2 \\
\alpha < \beta < \gamma = \delta \;\;\;\text{and}\;\;\;
\alpha \not\equiv \beta \equiv \gamma \equiv \delta \mod 2
\end{array}
\right.\right\}\;.
$$
The pairs that contribute to $c_2$ as in Lemma~\ref{c2formula} are $(p_{\alpha\gamma}, p_{\beta\delta})$ with $(\alpha, \beta, \gamma, \delta) \in \mathcal{Q}$, where $p_{\alpha\gamma}$ is ascending and $p_{\beta\delta}$ is descending.

For a crossing in a given diagram, the question whether it is descending or ascending is coupled to the question whether it is positive or negative. For star diagrams, if $\alpha < \gamma$ and $\pi(\alpha) < \pi (\gamma)$ then the sign of the crossing $p_{\alpha\gamma}$ is given by $(-1)^{\alpha+\gamma}$. Similarly, if $\beta < \delta$ and $\pi(\beta) > \pi (\delta)$ then $\text{sign}(p_{\beta\delta}) = (-1)^{\beta+\delta+1}$. By Lemma~\ref{c2formula} the Casson invariant is
$$ c_2(K_{2n+1}(\pi)) \;=\; \sum\limits_{(\alpha,\beta,\gamma,\delta) \in \mathcal{Q}} (-1)^{\alpha+\beta+\gamma+\delta+1} \cdot \mathcal{I}\left[\substack{\pi(\alpha) < \pi(\gamma) \\ \pi(\delta) < \pi(\beta)}\right] $$
where, as before, $\mathcal{I}[condition]=1$ if the condition holds, and else $0$.

\subsection*{Positive Expectations}\label{Positive-Expectations}

As a warm-up to later calculations we evaluate the first moment $E[c_2]$. By linearity, it is a signed sum of probabilities,
$$ E\left[c_2(K_{2n+1}(\pi))\right] \;=\;
\sum\limits_{(\alpha,\beta,\gamma,\delta) \in \mathcal{Q}} (-1)^{\alpha+\beta+\gamma+\delta+1} \cdot
P\left[\substack{\pi(\alpha) < \pi(\gamma) \\ \pi(\delta) < \pi(\beta)}\right] $$
The probability that $\pi(\alpha) < \pi(\gamma)$ and $\pi(\delta) < \pi(\beta)$ for a random $\pi \in S_{2n+1}$ depends on the quadruple's type. For terms of the first type, where $\alpha<\beta<\gamma<\delta$, it involves four distinct entries of $\pi$, a uniformly sampled permutation. Since $P[\pi(\alpha) < \pi(\gamma)] = P\left[\pi(\delta) < \pi(\beta)\right] = 1/2$ independently, the probability of the conjunction is $1/4$. However, if $\alpha=\beta$ then the probability of $\pi(\delta) < \pi(\beta) = \pi(\alpha) < \pi(\gamma)$ is only $1/6$, since all six orderings of three entries of $\pi$ are equally likely. Similarly for $\gamma=\delta$. For the case $\beta=\gamma$, the probability of $\pi(\alpha) < \pi(\gamma) = \pi(\beta) > \pi(\delta)$ is $1/3$, as a specific one of three entries has to be the highest.

It remains to count with signs the terms of each type. Suppose that the segments $\alpha,\beta,\gamma,\delta$ are distinct in the range $\{0,\dots,2n\}$ in which $n+1$ segments are even and $n$ are odd. Splitting the sum of signs according to how many segments among the four are even, we obtain
$$ \sum\limits_{\substack{(\alpha,\beta,\gamma,\delta) \in \mathcal{Q} \\ \alpha<\beta<\gamma<\delta}} (-1)^{\alpha+\gamma+\beta+\delta+1} \;=\; -\tbinom{n}{4} +\tbinom{n}{3}(n+1) -\tbinom{n}{2}\tbinom{n+1}{2} +n\tbinom{n+1}{3} -\tbinom{n+1}{4} \;=\; -\tbinom{n}{2} \;.$$
When $\beta=\gamma$, the sum is over triples of segments of the same parity, either even or odd, by definition of $\mathcal{Q}$. Hence,
$$ \sum\limits_{\substack{(\alpha,\beta,\gamma,\delta) \in \mathcal{Q} \\ \beta=\gamma}} (-1)^{\alpha+\gamma+\beta+\delta+1} \;=\; - \tbinom{n+1}{3} - \tbinom{n}{3} \;.$$
Counting $\alpha = \beta < \gamma < \delta$ such that $\delta$ is odd and the rest are even, is equivalent to counting triples $\{\alpha,\gamma,\delta+1\}$ of distinct even numbers $\leq 2n$. So is the case of complementary parities, with triples $\{\alpha-1,\gamma-1,\delta\}$. Together with the similar case of $\gamma = \delta$, this yields
$$ \sum\limits_{\substack{(\alpha,\beta,\gamma,\delta) \in \mathcal{Q} \\ \alpha=\beta}} (-1)^{\alpha+\gamma+\beta+\delta+1} \;+\; \sum\limits_{\substack{(\alpha,\beta,\gamma,\delta) \in \mathcal{Q} \\ \gamma=\delta}} (-1)^{\alpha+\gamma+\beta+\delta+1} \;=\; 4 \tbinom{n+1}{3} \;.$$
In conclusion,
\begin{equation}\label{expectation}
E[c_2] \;=\; \frac{-\tbinom{n}{2}}{4} \;+\; \frac{-\tbinom{n+1}{3} - \tbinom{n}{3}}{3} \;+\; \frac{4\tbinom{n+1}{3}}{6} \;=\; \frac{\tbinom{n}{2}}{12} \;,
\tag{$\star\star$}
\end{equation}
proving Theorem~\ref{moments2} for $k=1$.
This yields the convergence of the normalized expectation,
$$E\left[\frac{c_2(K_{2n+1})}{n^2}\right] \;\xrightarrow {n \rightarrow \infty}\; \frac{1}{24}\;.$$
This result already calls for a few remarks. 
\begin{itemize}
\item 
The positivity of the expected Casson invariant in this model, both for each $n$ and in the limit, is intriguing. It is unclear whether this is an artifact of the model, but we will give experimental evidence that it occurs in other models as well. 
\item
Although there are $\Theta(n^4)$ terms in the sum leading to the expectation, its value is of order $n^2$. A similar result for all moments is the content of Theorem~\ref{moments2}. This should not be taken for granted. In a slightly modified random model that we call the star model and discuss below, the expectation is of order~$n^3$.
\item
The maximum of $c_2$ over all diagrams with $2n+1$ petals has nevertheless order $n^4$. Indeed, as shown in \cite{adams2012knot}, the permutation defined by
$$ \pi(k) \;=\; nk \mod 2n+1 $$
yields the $(n,n+1)$ torus knot, and $c_2(\pi) = \tbinom{n+2}{4} \approx n^4 / 24$ since the Casson invariant of the~$(p,q)$~torus knot is given by $(p^2-1)(q^2-1)/24$~\cite{alvarez1996vassiliev}. 
\item
Despite its positive bias, the distribution of the Casson invariant reaches values of order $n^4$ also in its negative tail. A more involved permutation $\pi \in S_{2n+1}$, that we have constructed, but do not describe here, yields knots with $c_2(\pi) < -n^4/200$.
\end{itemize}

\subsection*{Proof of Theorem \ref{moments2}}

Consider the $k$th power of the Casson invariant of a random knot,
$$ c_2^k(K_{2n+1}(\pi)) \;=\;
\left(\sum\limits_{(\alpha,\beta,\gamma,\delta) \in \mathcal{Q}} (-1)^{\alpha+\beta+\gamma+\delta+1} \cdot
\mathcal{I}\left[\substack{\pi(\alpha) < \pi(\gamma) \\ \pi(\delta) < \pi(\beta)}\right]\right)^k \;\;. $$
Our aim is to show that its expectation is a polynomial in $n$ of degree at most $2k$. The proof has four main steps.

\begin{enumerate}[leftmargin=12pt]
\item \textbf{Patterns, Parities and Permutations}

We henceforth denote a quadruple in $\mathcal{Q}$ by $Q_i = (q_{i1}, q_{i2}, q_{i3}, q_{i4})$. A sequence of $k$ quadruples will be denoted by $\bold{Q} = (Q_1,\dots,Q_k) \in \mathcal{Q}^k$. Thus the above expression expands to
$$ c_2^k(K_{2n+1}(\pi)) \;=\; \sum\limits_{\bold{Q} \in \mathcal{Q}^k}
\;\prod\limits_{i=1}^{k} (-1)^{q_{i1}+q_{i2}+q_{i3}+q_{i4}+1} \cdot \mathcal{I}\left[\substack{\pi(q_{i1}) < \pi(q_{i3}) \\ \pi(q_{i4}) < \pi(q_{i2})}\right] \;.$$
We split the sum according to equalities and order relations between the $4k$ segments~$\{q_{ij}\}$, that occur in the $2k$ crossings. In order to encode this information, we make the following definition. A \emph{pattern} $\bold{T} = (T_1,\dots,T_k)$ is a sequence of quadruples of natural numbers, $T_i = (t_{i1},t_{i2},t_{i3},t_{i4})$, such that for some~$t \in \N$,
$$\bigcup\limits_{i=1}^k T_i = \{1,\dots,t\}\;.$$
Clearly $t = t(\bold{T}) = \max \{t_{ij}\} \leq 4k$.

Let $\bold{Q} = (Q_1,\dots,Q_k) \in \mathcal{Q}^k$. Then there exists a unique pattern $\bold{T}$ such that for each $i$ and~$j$,
$$q_{ij} < q_{i'j'} \;\Leftrightarrow\; t_{ij} < t_{i'j'} \;.$$
The pattern $\bold{T}$ is obtained from $\bold{Q}$ by the unique order preserving bijection between $\bigcup_i Q_i$ and $\{1,\dots,t\}$, eliminating the gaps between segment numbers. Note that $t_{i1} \leq t_{i2} \leq t_{i3} \leq t_{i4}$ for every $i$, with at most one equality. $\mathcal{T}_k$ denotes the collection of all patterns corresponding to $\mathcal{Q}^k$, for $n$ large enough. 

We further split the sum according to the parities of $\{q_{ij}\}$. This is encoded by a function $\varepsilon:[t]\rightarrow\nobreak\pm1$, where $\varepsilon(t_{ij}) = (-1)^{q_{ij}}$, or in short by a vector $\varepsilon \in \{\pm \}^t$.

Given a pattern $\bold{T} \in \mathcal{T}_k$ and a parity vector $\varepsilon \in \{\pm \}^t$ where $t = t(\bold{T})$, we denote by $\mathcal{Q}^k(\bold{T},\varepsilon)$ the set of all $\bold{Q} \in \mathcal{Q}^k$ with pattern $\bold{T}$ and parities $\varepsilon$.

The rearranged sum is
$$ c_2^k(K_{2n+1}(\pi)) \;=\; \sum\limits_{\bold{T} \in \mathcal{T}_k}\; \sum\limits_{\varepsilon \in \{\pm\}^t}\;\sum\limits_{\bold{Q} \in \mathcal{Q}^k(\bold{T},\varepsilon)}(-1)^k\prod\limits_{i=1}^{k} \varepsilon(t_{i1})\varepsilon(t_{i2})\varepsilon(t_{i3})\varepsilon(t_{i4}) \cdot \mathcal{I}\left[\substack{\pi(q_{i1}) < \pi(q_{i3}) \\ \pi(q_{i4}) < \pi(q_{i2})}\right] \;.$$

The expectation of $c_2^k$ is the average of the above sum over $\pi \in S_{2n+1}$. By linearity, it can be computed for every $\bold{Q} \in \mathcal{Q}^k$ separately. Observe that, given $\bold{Q}$ with pattern $\bold{T}$, only the order among the $t(\bold{T})$ entries $\{\pi(q_{ij})\}$ affects that term. Each $\pi \in S_{2n+1}$ induces a unique permutation $\sigma \in S_t$ such that
$$\pi(q_{ij}) < \pi(q_{i'j'}) \;\Leftrightarrow\; \sigma(t_{ij}) < \sigma(t_{i'j'})\;.$$
Thus we may average over $\sigma$ instead of $\pi$. By symmetry, each $\sigma \in S_t$ has equal weight~$1/t!$.
$$ E[c_2^k] \;=\; \sum\limits_{\bold{T} \in \mathcal{T}_k}\; \sum\limits_{\varepsilon \in \{\pm\}^t}\; \sum\limits_{\bold{Q} \in \mathcal{Q}^k(\bold{T},\varepsilon)} (-1)^k\;\frac{1}{t!} \sum\limits_{\sigma \in S_t} \;\prod\limits_{i=1}^{k} \varepsilon(t_{i1})\varepsilon(t_{i2})\varepsilon(t_{i3})\varepsilon(t_{i4}) \cdot \mathcal{I}\left[\substack{\sigma(t_{i1}) < \sigma(t_{i3}) \\ \sigma(t_{i4}) < \sigma(t_{i2})}\right] \;.$$
For fixed pattern, parities and permutation, every $\bold{Q} \in \mathcal{Q}^k(\bold{T},\varepsilon)$ contributes the same term from $\{-1,0,1\}$ to the above sum. We hence rewrite
$$ E[c_2^k] \;=\; \sum\limits_{\bold{T} \in \mathcal{T}_k}\; \sum\limits_{\varepsilon \in \{\pm\}^t}\; \sum\limits_{\sigma \in S_t} \frac{(-1)^k}{t!} \; \left|\mathcal{Q}^k(\bold{T},\varepsilon)\right| \;\prod\limits_{i=1}^{k} \varepsilon(t_{i1})\varepsilon(t_{i2})\varepsilon(t_{i3})\varepsilon(t_{i4}) \cdot \mathcal{I}\left[\substack{\sigma(t_{i1}) < \sigma(t_{i3}) \\ \sigma(t_{i4}) < \sigma(t_{i2})}\right] \;.$$
Note that the number of terms in the current sum is a function of $k$ and not of $n$. The dependence on $n$ is only through the factors $\left|\mathcal{Q}^k(\bold{T},\varepsilon)\right|$.

\item \textbf{Counting}

What is the size of $\mathcal{Q}^k(\bold{T},\varepsilon)$? Each $\bold{Q} \in \mathcal{Q}^k(\bold{T},\varepsilon)$ is determined uniquely by the $t$~numbers in $\{0,\dots,2n\}$ appearing in it. But not every one of the $\tbinom{2n+1}{t}$ options has the desired parity vector $\varepsilon \in \{\pm\}^t$. The following lemma counts how many do.

Denote by $z(\varepsilon_1,\dots,\varepsilon_t)$ the number of ``+" runs in $\varepsilon$. For example,
$$z(---) = 0, \;\; z(+++-) = 1, \;\; z(-+--+) = 2.$$

\begin{lemma}\label{runs}
Let $\varepsilon \in \{\pm\}^t$. Then $\;\# \{0 \leq q_1 < \cdots < q_t \leq 2n \;|\; \forall i\; (-1)^{q_i} = \varepsilon_i \} \;=\; \binom{n + z(\varepsilon)}{t}$.
\end{lemma}

\begin{proof}
We bijectively associate to every sequence $0 \leq q_1 < \dots < q_t \leq 2n$ of given parity $\varepsilon_1,\dots,\varepsilon_t$, a sequence of odd numbers $0 < q_1' < \dots < q_t' < 2(n+z)$, where $z=z(\varepsilon_1,\dots,\varepsilon_t)$. There are $n+z$ odd numbers in this range, and so the lemma follows.

Note that each run of pluses in $\varepsilon$ comes from a block of even numbers in $(q_1,\ldots,q_t)$, and minus runs come from odd blocks. To define the bijection, increase each element of the first block of even numbers by one. Then increase the elements of the following block, of odd numbers, by two, and so on. If the first block is odd, then its elements stay put. Explicitly, if $q_i$ is in block number $j$, then $q_i'=q_i+j$ where the numbering $j$ starts from $0$ or $1$ according to whether $q_1$ is odd or even.

By construction, all $q_i'$ are odd and vary between $0$ and $2(n+z)$. Indeed, numbers in the last block are increased by either $2z$ or $2z-1$, and $q_t$ may reach up to $2n-1$ or $2n$, according to whether it is odd or even.
\end{proof}

By our construction of $\mathcal{Q}$, equalities between elements of a quadruple $Q_i \in \mathcal{Q}$ impose restrictions on the parities in $Q_i$. Therefore, a pattern $\bold{T} \in \mathcal{T}_k$ may be incompatible with some parity vectors $\varepsilon$, in which case $\mathcal{Q}^k(\bold{T},\varepsilon)$ is empty. The following factor filters out such incompatible combinations.
$$ f(T_i,\varepsilon) =
\begin{cases}
1 & t_{i1} < t_{i2} < t_{i3} < t_{i4} \\
1 & t_{i1} = t_{i2} \text{ and } \varepsilon(t_{i1}) = \varepsilon(t_{i2}) = \varepsilon(t_{i3}) \neq \varepsilon(t_{i4}) \\
1 & t_{i2} = t_{i3} \text{ and } \varepsilon(t_{i1}) = \varepsilon(t_{i2}) = \varepsilon(t_{i3}) = \varepsilon(t_{i4}) \\
1 & t_{i3} = t_{i4} \text{ and } \varepsilon(t_{i1}) \neq \varepsilon(t_{i2}) = \varepsilon(t_{i3}) = \varepsilon(t_{i4}) \\
0 & \text{else} \end{cases} \;\;. $$
The count $\left|\mathcal{Q}^k(\bold{T},\varepsilon)\right|$ of compatible $\bold{Q}$ is obtained by combining Lemma~\ref{runs} with $f$:
$$ \left|\mathcal{Q}^k(\bold{T},\varepsilon)\right| \;=\; \binom{n + z(\varepsilon)}{t} \prod\limits_{i=1}^k f(T_i,\varepsilon) \;.$$
The expectation now becomes
$$ E\left[c_2^k\right] \;=\; \sum\limits_{\bold{T} \in \mathcal{T}_k} \sum\limits_{\varepsilon \in \{\pm\}^t} \sum\limits_{\sigma \in S_t} \frac{(-1)^k}{t!} \tbinom{n + z(\varepsilon)}{t} \prod\limits_{i=1}^k f(T_i,\varepsilon) \cdot \varepsilon(t_{i1})\varepsilon(t_{i2})\varepsilon(t_{i3})\varepsilon(t_{i4}) \cdot \mathcal{I}\left[\substack{\sigma(t_{i1}) < \sigma(t_{i3}) \\ \sigma(t_{i4}) < \sigma(t_{i2})}\right] .$$
This is a polynomial in $n$ of degree at most $4k$, since $t \leq 4k$. We will reduce it to $2k$.

\item \textbf{Exchange between Patterns}

It turns out that the factors $f(T_i,\varepsilon)$ can be replaced with simpler ones,
$$ F(T_i) \;=\; \begin{cases} 1 & t_{i1} < t_{i2} < t_{i3} < t_{i4} \\ \frac12 & \text{else, i.e., if there is one equality} \end{cases} \;\;\;.$$
This involves transfer of mass between terms in the sum. The key step is the following calculation.

\begin{lemma}\label{masstransfer}
Fix $t\in\N$, $\sigma \in S_t$ and $\varepsilon \in \{\pm\}^t$. For $1 \leq x<y<z \leq t$ consider the set
$$ \mathcal{U} = \{(x,x,y,z),(x,y,y,z),(x,y,z,z)\} $$
Then
\begin{align*}
\sum\limits_{S \in \mathcal{U}} f(S,\varepsilon) \cdot &\varepsilon(s_{1})\varepsilon(s_{2})\varepsilon(s_{3})\varepsilon(s_{4})
\cdot \mathcal{I}\left[\substack{\sigma(s_{1}) < \sigma(s_{3}) \\ \sigma(s_{4}) < \sigma(s_{2})}\right] \\ \;=\;&
\sum\limits_{S \in \mathcal{U}} F(S) \cdot \varepsilon(s_{1})\varepsilon(s_{2})\varepsilon(s_{3})\varepsilon(s_{4})
\cdot \mathcal{I}\left[\substack{\sigma(s_{1}) < \sigma(s_{3}) \\ \sigma(s_{4}) < \sigma(s_{2})}\right]
\end{align*}
\end{lemma}

\begin{proof}
We show that the two sides agree
\begin{eqnarray*}
\text{left hand side}
&=&\phantom{+\;}\mathcal{I}\left[\varepsilon(x) = \varepsilon(y) \neq \varepsilon(z)\right] \cdot \varepsilon(y)\varepsilon(z)
\cdot \mathcal{I}\left[\sigma(z) < \sigma(x) < \sigma(y)\right] \\
&&+\;\mathcal{I}\left[\varepsilon(x) = \varepsilon(y) = \varepsilon(z)\right] \cdot \varepsilon(x)\varepsilon(z)
\cdot \mathcal{I}\left[\sigma(x) < \sigma(y),\; \sigma(z) < \sigma(y)\right] \\
&&+\;\mathcal{I}\left[\varepsilon(x) \neq \varepsilon(y) = \varepsilon(z)\right] \cdot \varepsilon(x)\varepsilon(y)
\cdot \mathcal{I}\left[\sigma(x) < \sigma(z) < \sigma(y)\right] \\
&=& \begin{cases}
(\varepsilon(x) + \varepsilon(y))/2 \cdot \varepsilon(z)&\;\;\;\;\;\;\;\text{if }\sigma(z) < \sigma(x) < \sigma(y) \\
(\varepsilon(y) + \varepsilon(z))/2 \cdot \varepsilon(x)&\;\;\;\;\;\;\;\text{if }\sigma(x) < \sigma(z) < \sigma(y) \\
0 & \;\;\;\;\;\;\;\text{otherwise}
\end{cases} \\
&=&\phantom{+\;}\varepsilon(y)\varepsilon(z)/2 \cdot \mathcal{I}\left[\sigma(z) < \sigma(x) < \sigma(y)\right] \\
&&+\;\varepsilon(x)\varepsilon(z)/2 \cdot \mathcal{I}\left[\sigma(x) < \sigma(y),\; \sigma(z) < \sigma(y)\right] \\
&&+\;\varepsilon(x)\varepsilon(y)/2 \cdot \mathcal{I}\left[\sigma(x) < \sigma(z) < \sigma(y)\right] \\
&=& \text{right hand side} .
\end{eqnarray*}
This is easily verified for any fixed $\sigma$ and $\varepsilon$, by considering various cases of $\varepsilon(x),\varepsilon(y),\varepsilon(z)$ and the ordering of $\sigma(x),\sigma(y),\sigma(z)$.
\end{proof}

We now apply Lemma~\ref{masstransfer} on the expression for $E\left[c_2^k\right]$. Let $\bold{T} = (T_1,\dots,T_k) \in \mathcal{T}_k$. If $T_1$ consists of four distinct numbers then $f(T_1,\varepsilon) = 1 = F(T_1)$. Otherwise, $\bold{T}$ is one of three possible patterns in $\mathcal{T}_k$, that agree on $T_2,\dots,T_k$, and for which $T_1$ is one of $\{(x,x,y,z)$, $(x,y,y,z)$, $(x,y,z,z)\}$ for some $x<y<z$. We thus apply the lemma and replace $f(T_1,\varepsilon)$ with $F(T_1)$ in the terms corresponding to these patterns. Note that for every fixed $\sigma$ and $\varepsilon$, the same multiplicative factor comes from $T_2,\dots,T_k$.

We thus apply the lemma for all such triples of patterns, and turn every $f(T_1,\varepsilon)$ into $F(T_1)$. We repeat for $T_2,\dots,T_k$, and so every factor $f(T_i,\varepsilon)$ is replaced by $F(T_i)$.
$$ E\left[c_2^k\right] \;=\; \sum\limits_{\bold{T} \in \mathcal{T}_k} \sum\limits_{\varepsilon \in \{\pm\}^t} \sum\limits_{\sigma \in S_t} \frac{(-1)^k}{t!} \tbinom{n + z(\varepsilon)}{t} \prod\limits_{i=1}^k F(T_i) \cdot \varepsilon(t_{i1})\varepsilon(t_{i2})\varepsilon(t_{i3})\varepsilon(t_{i4}) \cdot \mathcal{I}\left[\substack{\sigma(t_{i1}) < \sigma(t_{i3}) \\ \sigma(t_{i4}) < \sigma(t_{i2})}\right] .$$
We next define two functions of $\bold{T} = (T_1,\dots,T_k) \in \mathcal{T}_k$, that simplify our notation. Denote by $\Delta = \Delta(\bold{T}) \in \{0,\dots,k\}$ the number of quadruples $T_i$ with some equality in $t_{i1} \leq t_{i2} \leq t_{i3} \leq t_{i4}$, i.e., those with only three distinct elements. It is immediate from the definitions that
$$\frac{1}{2^{\Delta(\bold{T})}} \;=\; \prod\limits_{i=1}^k F(T_i) \;.$$
Let $P(\bold{T})$ denote the probability that a random permutation $\sigma \in S_t$ satisfies all the inequalities imposed by~$\bold{T}$:
$$ P(\bold{T}) \;=\; \frac{1}{t!} \sum\limits_{\sigma \in S_t} \;\prod\limits_{i=1}^k
\;\mathcal{I}\left[\substack{\sigma(t_{i1}) < \sigma(t_{i3}) \\ \sigma(t_{i4}) < \sigma(t_{i2})}\right] \;.$$
The expectation is now simplified to
$$ E\left[c_2^k\right] \;=\;
(-1)^k \sum\limits_{\bold{T} \in \mathcal{T}_k} \frac{P(\bold{T})}{2^{\Delta(\bold{T})}}
\sum\limits_{\varepsilon \in \{\pm\}^t} \binom{n + z(\varepsilon)}{t} \;\prod\limits_{i=1}^k
\varepsilon(t_{i1})\varepsilon(t_{i2})\varepsilon(t_{i3})\varepsilon(t_{i4}) \;.$$
Thanks to the inter-pattern exchanges, each pattern now separately contributes a polynomial in $n$ of degree $\leq 2k$, as we show next.

\item \textbf{Cancellations}

As in the case of the linking number, the sum over $\varepsilon$ vanishes under certain conditions. But we first state two more formulas:
\begin{enumerate}
\item\label{move-a}
Recall from the proof of Lemma~\ref{lnlemma} the orthogonal family of functions $\chi(I)(\varepsilon_1,\dots,\varepsilon_t) = \prod_{j \in I}\varepsilon_j$ where $I \subseteq \{1,\dots,t\}$. For $\bold{T} \in \mathcal{T}_k$, define $J = J(\bold{T})$ to be the set of numbers in $\{1,\dots,t\}$ that appear in $\bold{T}$ an odd number of times. Then
$$ \prod\limits_{i=1}^k \varepsilon(t_{i1})\varepsilon(t_{i2})\varepsilon(t_{i3})\varepsilon(t_{i4}) \;=\; \chi(J(\bold{T}))(\varepsilon) \;.$$
Note that $|J| \geq 2t-4k$, where in case of equality each element of $\{1,\dots,t\}$ appears in the sequence $\{t_{ij}\}$ either once or twice.
\item\label{move-b}
The following is a standard identity.
$$ \binom{n+z}{t} \;=\; \sum\limits_{r=0}^{t} \binom{n}{t-r} \binom{z}{r} \;\;.$$
On the left is the number of size $t$ subsets of $[n+z]$, and on the right the count is split according to the number of elements in the first $n$ positions.
\end{enumerate}
By~(a) and~(b),
$$ E\left[c_2^k\right] \;=\;
(-1)^k \sum\limits_{\bold{T} \in \mathcal{T}_k} \frac{P(\bold{T})}{2^{\Delta(\bold{T})}}
\;\sum\limits_{r=0}^{t} \;\binom{n}{t-r}
\;\sum\limits_{\varepsilon \in \{\pm\}^t} \binom{z(\varepsilon)}{r} \cdot \chi(J(\bold{T}))(\varepsilon) $$
Note also that
$$ z(\varepsilon_1,\dots,\varepsilon_t) \;=\; \frac{(t+1) + \varepsilon_1 - \varepsilon_1\varepsilon_2 - \varepsilon_2\varepsilon_3 - \dots - \varepsilon_{t-1}\varepsilon_t + \varepsilon_t}{4} \;\;.$$
The sum over $\varepsilon \in \{\pm 1\}^t$ can be viewed as the inner product $\left\langle \tbinom{z}{r},\chi(J)\right\rangle$ of two functions in $\mathcal{W}_t$, the linear space of functions $\{\pm 1\}^t \to \nobreak \R$. The terms in the sum over $\bold{T}$ and $r$ divide into four cases.
\begin{itemize}
\item
If $r > t-2k$ then $\tbinom{n}{t-r}$ is a polynomial in $n$ of degree $ < 2k$.
\item
If $r < t-2k$ then $\left\langle \tbinom{z}{r},\chi(J)\right\rangle = 0$. Indeed, by the above $z(\varepsilon)$ formula, $\tbinom{z(\varepsilon)}{r}$~is spanned by $\left\{\chi(I) : |I| \leq 2r\right\}$, while $|J| \geq 2t-4k > 2r$.
\item
If $r = t-2k$ and $|J| > 2t-4k$ then similarly the functions are orthogonal.
\item
If $r = t-2k$ and $|J| = 2t-4k$ then the term equals $\tbinom{n}{2k} \cdot \left\langle \tbinom{z}{r},\chi(J)\right\rangle$.
\end{itemize}
This shows that $E\left[c_2^k\right]$ has degree at most $2k$ as a polynomial in $n$, proving Theorem~\ref{moments2}.\qed 
\end{enumerate}

We continue the argument and evaluate the inner products in the last case of the four, thus deriving this polynomial's leading term. Let $|J| = 2t-4k = 2r$. We first consider the projection of $\tbinom{z}{r} \in \mathcal{W}_t$ to the subspace $\text{\emph{span}}\left\{\chi(I):|I| = 2r \right\}$, that contains $\chi(J)$. Observe that $\tbinom{z}{r}$ has the same projection as the function
$$ z_r(\varepsilon) \;:=\; \frac{(\varepsilon_1\varepsilon_2+\varepsilon_2\varepsilon_3+\dots+\varepsilon_{t-1}\varepsilon_t)^r}{(-4)^r \cdot r!} \;.$$
This function has a non-zero inner product with $\chi(J)$ for $|J|=2r$, if and only if $J$ is the disjoint union of $r$ pairs of consecutive numbers. In this case we conclude
$$ \sum\limits_{\varepsilon \in \{\pm\}^t} \binom{z(\varepsilon)}{r} \chi(J(\bold{T}))(\varepsilon) \;=\; \left\langle z_r(\varepsilon),\chi(J)\right\rangle\;=\; 2^t \cdot \frac{ r!}{(-4)^{r} r!} \;=\; (-1)^t 2^{4k-t} \;.$$
where we used $\left\langle \chi(J),\chi(J')\right\rangle = 2^t \delta_{J,J'}$.

Recall that if $|J(\bold{T})| = 2t-4k$ then every number appears in $\bold{T}$ at most twice, which means that $J(\bold{T})$ consists of those that appear once. Therefore the patterns that contribute $\tbinom{n}{2k}$ to the $k$th moment are captured by the following definition. Call a pattern $\bold{T} \in \mathcal{T}_k$ a \emph{principal} pattern if
\begin{enumerate}
\item
Every element of $\{1,\dots,t(\bold{T})\}$ appears in the sequence $\{t_{ij}\}$ at most twice.
\item
The set of elements that appear only once is a disjoint union of pairs of consecutive numbers. 
\end{enumerate}
For example, the pattern $\bold{T} = ((1,1,4,6),(2,3,4,5))$ is principal with $J(\bold{T}) = \{2,3,5,6\}$, while $((1,2,2,6),(3,4,5,5))$ is not principal, as $\{1,3,4,6\}$ is not a disjoint union of pairs of consecutive numbers.

Denoting the set of principal patterns by $\mathcal{T}_k^*$, we finally write
$$ E\left[c_2^k\right] \;=\; (-1)^k \binom{n}{2k} \;\sum\limits_{\bold{T} \in \mathcal{T}_k^*} (-1)^t \cdot 2^{4k-t-\Delta(\bold{T})} \cdot P(\bold{T}) \;+\; R(n) $$
where $R(n)$ is a polynomial of degree at most $2k-1$.

\subsection*{Formula for the limiting \emph{k}th moment}

Next, we simplify the formula for the limiting normalized moments. Denote by $\lambda_k$ the coefficient of $n^{2k}$ in the expected $c_2^k$.
$$ \lambda_k \;:=\; \lim_{n\to\infty} E\left[\frac{c_2(K_{2n+1})^k}{n^{2k}}\right] \;=\; \frac{(-1)^k}{(2k)!}\;\sum\limits_{\bold{T} \in \mathcal{T}_k^*} (-1)^t \cdot 2^{4k-t-\Delta(\bold{T})} \cdot P(\bold{T}) \;.$$
We say that $\bold{T}'$ is a \emph{refinement} of $\bold{T}$, writing $\bold{T}' \prec \bold{T}$, if
$$ t_{ij} < t_{lm} \;\Rightarrow\; t_{ij}' < t_{lm}' $$
This allows for $t_{ij}' < t_{lm}'$ while $t_{ij} = t_{lm}$. For example, $\bold{T}' = ((1,1,5,7),(2,3,4,6))$ refines $\bold{T} = ((1,1,4,6),(2,3,4,5))$. In other words, $\bold{T}$ is obtained from $\bold{T}'$ by merging consecutive numbers. Note that any refinement of a principal pattern is principal.

Denote by $\mathcal{T}_k^1 \subseteq \mathcal{T}_k^*$ the set of patterns that contain every element in $\{1,\dots,4k\}$ exactly once. Let $\bold{T}' \in \mathcal{T}_k^1$. Denote by $\overline{\bold{T}'}$ the pattern obtained from $\bold{T}'$ by merging each of the $2k$ pairs $(1,2),(3,4),\dots,(4k-1,4k)$ into one number $1,2,\dots,2k$, so that each of these $2k$ numbers appears twice in $\overline{\bold{T}'}$. For example if $\bold{T}' = ((1,3,4,5),(2,6,7,8))$ then $\overline{\bold{T}'} = ((1,2,2,3),(1,3,4,4))$.

For $\bold{T}' \in \mathcal{T}_k^1$, we define a real function $\psi$, as a sum over all $2^{2k}$ ways to refine $\overline{\bold{T}'}$:
$$ \psi(\bold{T}') \;:=\; \sum\limits_{\bold{T}' \prec \bold{T} \prec \overline{\bold{T}'}} (-1)^{t(\bold{T})}P(\bold{T}) \;\;.$$
The following lemma rewrites $\lambda_k$ as a sum of $\psi(\bold{T}')$'s.

\begin{lemma}\label{psi-lemma}
$$ \lambda_k \;=\; \frac{(-1)^k}{(2k)!} \;\sum\limits_{\bold{T}' \in \mathcal{T}_k^1} \psi(\bold{T}')\;.$$
\end{lemma}

\begin{proof}
This is a change of order of summation, where both sides are equal to a double sum over~$\bold{T}$ and~$\bold{T}'$.

In the definition of $\lambda_k$, each pattern $\bold{T} \in \mathcal{T}_k^*$ can be refined to $\bold{T}' \in \mathcal{T}_k^1$ in $2^{4k-t-\Delta(\bold{T})}$ ways. Indeed, each of the $4k-t$ elements that appear twice in $\bold{T}$ can be replaced by two suitable consecutive numbers in two orders, except for those $\Delta(\bold{T})$ numbers with both occurrences in the same quadruple. For them there is only one possible ordering, the one that keeps the quadruple increasing.

Moreover, if $\bold{T}' \prec \bold{T} \in \mathcal{T}_k^*$ then $\bold{T} \prec \overline{\bold{T}'}$. Indeed, being a principal pattern, $\bold{T}$ can be obtained from its refinement $\bold{T}'$ by merging pairs of numbers that are consecutive in $\{1,\dots,4k\}$, leaving the remaining ones in runs of even length. It follows that each $\bold{T} \in \mathcal{T}_k^*$ is determined by which ones of the $2k$ pairs $(1,2),(3,4),\dots$ are merged and which remain distinct.

Not necessarily all these $2^{2k}$ options give actual patterns $\bold{T} \in \mathcal{T}_k^*$ that appear in the original sum. Such a pattern consists of quadruples $t_{i1} \leq t_{i2} \leq t_{i3} \leq t_{i4}$ with at most one equality, but by merging pairs it might happen that $t_{i1} = t_{i2} < t_{i3} = t_{i4}$, in which case $\bold{T} \not\in \mathcal{T}_k$. Such terms do not contribute to $\psi(\bold{T}')$ because the definition of $P(\bold{T})$ contains the conditions $\sigma(t_{i1}) < \sigma(t_{i3})$ and $\sigma(t_{i2}) > \sigma(t_{i4})$ which imply $P(\bold{T})=0$.
\end{proof}

Our next goal is to derive a formula for $\psi(\bold{T}')$ for $\bold{T}' \in \mathcal{T}_k^1$. We translate the problem from the language of patterns to the realm of directed graphs.
\begin{itemize}
\item 
Given a pattern $\bold{T}$, define a directed graph $G(\bold{T})$, with $t$ vertices labeled $1,\dots,t$ and $2k$ edges $t_{i3} \to t_{i1}$ and $t_{i2} \to t_{i4}$. For $\bold{T}' \in \mathcal{T}_k^1$, $G(\bold{T}')$ is a disjoint union of $2k$ edges, while $G(\overline{\bold{T}'})$ has degree $2$ vertices and so is a disjoint union of cycles, whose edges are individually oriented.
\item 
The \emph{breaking} of a graph at a vertex $v$ is the operation of replacing $v$ by two or more disjoint vertices, each of which gets some of $v$'s edges. Recall that the summation in $\psi(\bold{T}')$ is over $2^{2k}$ patterns $\bold{T}$ such that $\bold{T}' \prec \bold{T} \prec \overline{\bold{T}'}$. Translating refinement into terms of directed graphs, it means that $G(\bold{T})$ is obtained from $G(\overline{\bold{T}'})$ by breaking it at a subset of its vertices. Note that each of the $2k$ vertices in the original union of cycles may be broken into two vertices of total degree $1$, so that cycles break into unions of paths.
\item 
Let $G$ be a directed graph on $t$ vertices. For convenience, we abuse notation and let $G$ also refer to the event that a permutation $\sigma \in S_t$ that assigns values to its vertices, respects the orientations of its edges, i.e., 
$$ u \;\bullet\!\!\!\rightarrow\!\!\!-\!\!\!\bullet v \;\;\Rightarrow\;\; \sigma(u) > \sigma(v)\;.$$
Then $P(G)$ denotes the probability of this event where $\sigma \in S_t$ is picked uniformly at random. From the definitions $P(G(\bold{T})) = P(\bold{T})$. 
\item 
It follows that instead of summing over all $2^{2k}$ patterns between $\bold{T}'$ and $\overline{\bold{T}'}$, we may sum over all $2^{2k}$ breakings of $G(\overline{\bold{T}'})$. For a $2$-regular directed graph $G$ let
$$ \psi(G) \;:=\; \sum\limits_{H \in \mathcal{B}(G)}(-1)^{|H|}P(H)$$
where $\mathcal{B}(G)$ are all its $2^{|G|}$ breakings, and $|H|$ stands for the number of vertices in $H$. Clearly, $\psi(\bold{T}') = \psi(G)$ where $G = G(\overline{\bold{T}'})$.
\end{itemize}

\begin{lemma}\label{cycle}
Let $G$ be a $2k$-vertex $2$-regular directed graph, and denote the cycles in~$G$ by $C_1$, $\dots$, $C_m$, with $l_1,\dots,l_m$ vertices respectively, with each $l_i \geq 2$. Then
$$ \psi(G) \;=\; (-1)^k \; \prod_{j=1}^m \; \mathrm{sign}(C_j) \cdot \beta_{l_j} $$
where
\begin{itemize}
\item The sign $\pm 1$ of a cycle of even length is the parity of the number of forward edges encountered going once around the cycle. The sign of an odd cycle is $0$.
\item $\beta_l \;=\; (-1)^{l/2+1}B_l/l!$ for $l \geq 2$, where $B_l$ is the $l$th Bernoulli number~\cite[p. 1040]{jeffrey2007table}.
\end{itemize}
\end{lemma}

\begin{rmk*}
By the properties of the Bernoulli numbers, $\beta_l = 0$ for every odd $l \geq 3$. We arbitrarily set $\beta_0 = \beta_1 = 0$ as well.
\end{rmk*}

\begin{proof}
The proof has three parts. We first show that every cycle can be treated separately. Then we reduce to the case of a \emph{consistently oriented cycle}, that is a cycle in the directed sense, with one ingoing and one outgoing edge at each vertex. For consistently oriented cycles, we calculate~$\psi$ directly.

\begin{enumerate}
\item 
Let $H$ be a breaking of $G$. Note that the order relations between the value of $\sigma$ on a subset of the vertices is independent of the order relations within any disjoint subset of vertices. In particular, events that involve edges in $H$ that come from different cycles of $G$ are independent. Therefore, for such $H$,
$$ P(H) = P(H_1) P(H_2) \cdots P(H_m) $$ 
where $H_i$ is the subgraph of $H$ with vertices that come from the cycle $C_i$ in $G$. Now, by the independence of cycles and the distributive law
$$ \psi(G) = \psi(C_1) \psi(C_2) \cdots \psi(C_m) \;.$$
It is now sufficient to show that $\psi(C) = (-1)^{l/2} \mathrm{sign}(C) \beta_{l}$ on cycles of even length $l$, and $0$ on odd ones.

\item 
We further reduce to the case of a consistently oriented cycle. We show that flipping the orientation of an edge $e$ in $C$ changes the sign of $\psi(C)$. The lemma follows since this also flips $\text{sign}(C)$.

Let $C'$ be the cycle with a flipped edge, and let $v$ be a vertex incident to $e$. Recall that $\mathcal{B}(C)$ denotes all $2^{|C|}$ ways to break $C$, and denote by $\mathcal{B}_v(C)$ all $2^{|C|-1}$ ways to break $C$ except at $v$. For $H \in \mathcal{B}_v(C)$, we write $P(drawing)$, referring to $P(\tilde{H})$, where the drawing describes a neighborhood of $v$ where $\tilde{H}$ disagrees with~$H$, e.g. in the breaking of $v$ or in the orientation of $e$. With this notation,
\begin{align*}
\psi(C) + \psi(C') \;=\;
& \sum\limits_{H \in \mathcal{B}(C)} (-1)^{|H|} P(H) + \sum\limits_{H \in \mathcal{B}(C')} (-1)^{|H|} P(H) \\
\;=\;
& \sum\limits_{H \in \mathcal{B}_v(C)} (-1)^{|H|}
\left[P(-\!\!\bullet\!\!\!\rightarrow\!\!\!-)
- P(-\!\!\!\bullet\bullet\!\!\!\rightarrow\!\!\!-)
+ P(-\!\!\!\bullet\!\!\!-\!\!\!\leftarrow)
- P(-\!\!\!\bullet\bullet\!\!-\!\!\!\leftarrow) \right] \;\;.
\end{align*}
Note that summing the probabilities of two graphs that differ by the orientation of one edge yields the probability of the graph with that edge deleted. In terms of drawings,
$$
P(\cdots\bullet\!\!\!\rightarrow\!\!\!-\!\!\!\bullet\cdots)
+ P(\cdots\bullet\!\!\!-\!\!\!\leftarrow\!\!\!\bullet\cdots) \;=\;
P(\cdots\bullet\;\;\;\bullet\cdots) \;\;.
$$
By applying this to the first and third terms, and to the second and fourth terms, in the above sum,
$$
\psi(C) + \psi(C') \;=\;
\sum\limits_{H \in \mathcal{B}_v(C)} (-1)^{|H|}
\left[P(-\!\!\!\bullet\;\;\;\;) - P(-\!\!\!\bullet\bullet\;\;\;\;) \right] \;=\; 0 $$
Each term in the above sum vanishes since the addition of an isolated vertex preserves the probability of an oriented graph.

It follows that flipping all edges of a consistently oriented odd cycle negates $\psi$ and preserves the graph.
We deduce that $\psi(C)$ vanishes on any cycle $C$ of odd length $l > 1$.

\item 
We finally compute $\psi(C^{2k})$ where $C^l$ is the consistently oriented cycle on $l$ vertices. A nontrivial breaking of $C^l$ is a disjoint union of consistently oriented paths. For $P^i$, a path with $i$ edges, clearly $P(P^i) = 1/(i+1)!$. If $C^l$ is broken into $j$ paths, then they give independent events, and the probability is a product of such factors.

In order to sum over all breakings, we consider all ordered partitions $i_1+\dots+i_j=l$. Such partition corresponds to breaking $C^l$ into $j$ paths with a choice of which one is considered to be first. We thus multiply by $l$ as the first path can start at any point in $C^l$, and divide by~$j$ as any such partition is counted $j$ times.
$$ \psi(C^l) \;=\; \sum\limits_{j=1}^{l} \;\frac{l}{j} \sum\limits_{i_1+\dots+i_j=l} \frac{(-1)^j}{(i_1+1)! \cdots (i_j+1)!} $$
We define a generating function
$$ \Psi(x) \;=\; \sum\limits_{l=2}^{\infty} \psi(C^l) x^l \;=\;
\frac{x}{2} + \sum\limits_{l=1}^{\infty} \;\sum\limits_{j=1}^{l} \frac{(-1)^j}{j} \sum\limits_{i_1+\dots+i_j=l} \;\;\frac{l \cdot x^l}{(i_1+1)! \cdots (i_j+1)!} $$
where the term $x/2$ cancels the remaining odd case $l=1$. In order to identify $\Psi(x)$ with an analytic expression, we define
\begin{align*}
y(x) &\;=\; \frac{\exp(x)-1}{x}-1 \;=\; \sum\limits_{i=1}^{\infty} \frac{x^i}{(i+1)!} \\
z(y) &\;=\; -\log(1+y) \;=\; \sum\limits_{j=1}^{\infty} \frac{(-1)^j}{j} y^j
\end{align*}
which yields,
\begin{align*}
z(y(x)) &\;=\; -\log\left(\frac{\exp(x)-1}{x}\right) \;=\; \sum\limits_{j=1}^{\infty} \frac{(-1)^j}{j} \left(\sum\limits_{i=1}^{\infty} \frac{1}{(i+1)!} x^i\right)^j \\
 &\;=\; \sum\limits_{j=1}^{\infty} \;\sum\limits_{l=j}^{\infty} \frac{(-1)^j}{j}\sum\limits_{i_1+\dots+i_j=l} \;\;\frac{x^l}{(i_1+1)! \cdots (i_j+1)!} 
\;=\; \int_0^x\frac{\Psi(x')dx'}{x'} - \frac{x}{2} \;.
\end{align*}
By differentiation and using the power series for $\coth$~\cite[p. 42]{jeffrey2007table},
$$ \Psi(x) \;=\; x \left(\frac12 + \frac{dz}{dx}\right) \;=\; 1 - \frac{x}{2} \coth \frac{x}{2} \;=\; 
-\sum\limits_{l=2}^{\infty} \frac{B_l}{l!}x^l \;=\; \sum\limits_{l=0}^{\infty} (-1)^{l/2}\beta_l x^l \;.$$
In conclusion, $\psi(C^{2k}) = (-1)^k\beta_{2k}$ and the proof is complete. \qedhere
\end{enumerate}
\end{proof}

\begin{rmk*}
The first few $\beta_{2k}$'s are given by
$$ \beta_2 \;=\; \frac{1}{12} \;,\;\;\;\;
\beta_4 \;=\; \frac{1}{720} \;,\;\;\;\; 
\beta_6 \;=\; \frac{1}{30240} \;,\;\;\;\; 
\beta_8 \;=\; \frac{1}{1209600} \;,\;\;\;\; 
\beta_{10} \;=\; \frac{1}{47900160} \;.$$
\end{rmk*}

\begin{cor}\label{formula}
$$ \lambda_k \;=\; \frac{1}{(2k)!} \;\sum\limits_{\bold{T}' \in \mathcal{T}_k^1} \;\prod_{j=1}^m \mathrm{sign}(C_j) \cdot \beta_{l_j} \;,$$
where $C_1,\dots,C_m$ are the cycles in $G(\overline{\bold{T}'})$, of length $l_1,\dots,l_m$.
\end{cor}

\begin{proof}
This follows from Lemmas~\ref{psi-lemma} and~\ref{cycle}.
\end{proof}

\newcommand{\arcsgraph}[2]
{\tikzstyle{every node}=[circle, draw, minimum width=2pt]
\begin{tikzpicture}[baseline=-4pt]
\foreach \x in {1,...,#1} \filldraw (\x*0.4,0) circle (1pt);
\foreach \x/\y/\z in {#2} \draw[shorten >=1pt,-,decoration={markings,mark=at position 0.8 with {\arrow{>}}},postaction={decorate}] (\x*0.4,0) .. controls +(90-\z:0.2) and +(90+\z:0.2) .. (\y*0.4,0) ;
\end{tikzpicture}}

We demonstrate this formula on the first few moments.
\begin{itemize}
\item 
For $k=1$ the only pattern in $\mathcal{T}_1^1$ is $\bold{T}' = ((1,2,3,4))$, so that $\overline{\bold{T}'} = ((1,1,2,2))$, and $G(\overline{\bold{T}'}) = \arcsgraph{2}{1/2/45,2/1/-135}
$, a single positive cycle of length $2$. This yields
$$ E[c_2/n^2] \;\to\; \lambda_1 \;=\; \frac{+\beta_2}{2!} \;=\; \frac{1}{24} \;,$$
in accordance with our direct calculation in~(\ref{expectation}): $E[c_2] = n(n-1)/24$.

\item 
For $k=2$ there are $\tbinom{8}{4} = 70$ relevant patterns in $\mathcal{T}_2^1$, as $\{1,\dots,8\}$ should split between two quadruples. Sorting them into unions of cycles,
$$ E[c_2^2/n^4] \;\to\; \lambda_2 \;=\; \frac{6 (+\beta_2)^2 + 16 (-\beta_2)^2 + 32 (+\beta_4) + 16 (-\beta_4)}{4!} \;=\; \frac{7}{960}\;.$$
Some representative terms are: 
\begin{center}
\begin{tabular}{cccc}
$\bold{T}'$ & $\overline{\bold{T}'}$ & $G(\overline{\bold{T}'})$ & $\psi(G(\overline{\bold{T}'}))$ \\
\hline\hline
$((1,2,3,4),(5,6,7,8))$ & $((1,1,2,2),(3,3,4,4))$ & $\arcsgraph{4}{1/2/45,2/1/-135,3/4/45,4/3/-135}$ & $(+\beta_2)^2$ \\
$((1,3,5,8),(2,4,6,7))$ & $((1,2,3,4),(1,2,3,4))$ & $\arcsgraph{4}{3/1/-45,3/1/-135,2/4/45,2/4/135}$ & $(-\beta_2)^2$ \\
$((1,2,6,7),(3,4,5,8))$ & $((1,1,3,4),(2,2,3,4))$ & $\arcsgraph{4}{3/1/-45,3/2/-90,1/4/135,2/4/45}$ & $+\beta_4$ \\
$((1,6,7,8),(2,3,4,5))$ & $((1,3,4,4),(1,2,2,3))$ & $\arcsgraph{4}{2/1/-45,2/3/45,3/4/45,4/1/-135}$ & $-\beta_4$ \\
\hline\hline
\end{tabular}
\end{center}

\item 
The case $k=3$ was obtained with a computer.
$$ E[c_2^3/n^6] \;\to\; \lambda_3 \;=\; \frac{1194 \beta_2^3 + 5328 \beta_2\beta_4 + 6528\beta_6}{6!} \;=\; \frac{5119}{2419200} $$
In order to independently verify this outcome, we compute the entire distribution of \linebreak $c_2(K_{2n+1})$ for each $0 \leq n \leq 6$. By interpolation we obtain the moments as polynomials in $n$, as follows.
\begin{align*}
E[c_2(K_{2n+1})^2] \;&=\; \frac{7n^4 - 2n^3 - 3n^2 - 2n}{960} \\
E[c_2(K_{2n+1})^3] \;&=\; \frac{5119n^6-3033n^5-3125n^4+3465n^3-914n^2-1512n}{2419200}
\end{align*}
Note that the leading terms' coefficients are exactly $\lambda_2$ and $\lambda_3$.

\item 
For $k=4$,
$$\lambda_4 \;=\; \frac{194904\beta_2^4+1855872\beta_2^2\beta_4+4442112\beta_2\beta_6+1774080\beta_4^2+6506496\beta_8}{8!} \;=\; \frac{812143}{677376000} $$
\end{itemize}

Unfortunately, without better control of the cancellations in the Corollary~\ref{formula} sum, we cannot infer weak convergence of the normalized distributions. However, we see evidence for convergence in the histograms of $c_2/n^2$ for random samples of permutations, as in Figure~\ref{fig_hist}. These seem to converge as $n$ grows, to an asymmetric continuous distribution. 

\begin{figure}[bht]
\centering
\includegraphics[width=0.8\textwidth]{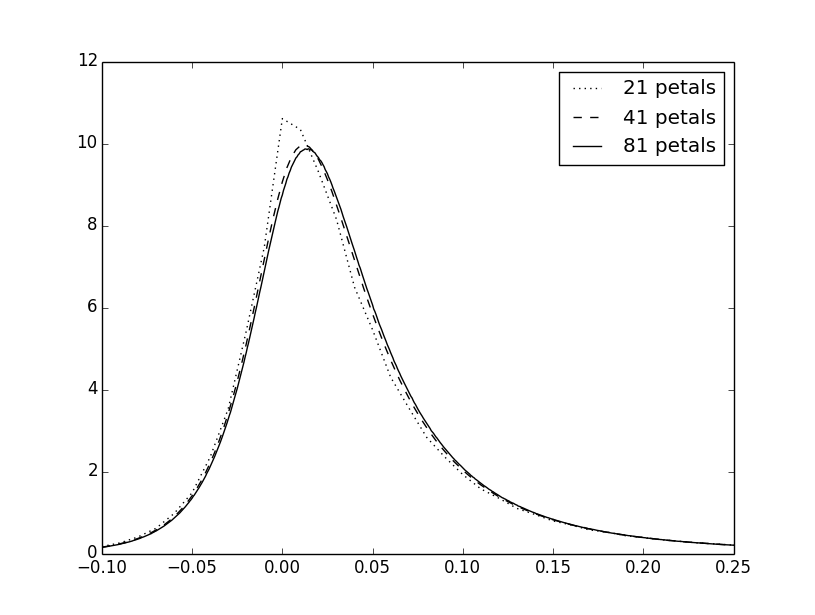}
\caption{Distribution of $c_2/n^2$ for $(2n+1)$-petal diagrams, based on over $10^9$ random samples.}
\label{fig_hist}
\end{figure}

\subsection*{Other Models}

To study model dependence, we compare our results for $c_2$ with two related random models, the star and the grid.

Our computations in the Petaluma model reduced the sum of $n^4$ terms to $O(n^2)$. This is significant because similar computations in the star model yield $n^3$. The \emph{star model} is defined by taking the $(2n+1)$-star diagram, as in Figure~\ref{fig_star9}, and choosing the sign of every crossing independently at random. This model is universal, since every petal diagram corresponds to a star diagram. However, the $(2n+1)$-star model realizes knots that a $(2n+1)$-petal diagram doesn't, like the $(2n+1,n)$ torus knot in the case that all crossings are positive.

One can compute the expectation of $c_2$ in this model, by Equation~(\ref{expectation}) with the probabilities $1/3$ and $1/6$ replaced by $1/4$. This yields $E[c_2] = (n^3-n)/12 = \Theta(n^3)$, compared to $\Theta(n^2)$ in the Petaluma model. By a similar computation for the second moment the variance is given by $(2n^4 + n^3 - 2n^2 - n)/48 = O(n^4)$. By Chebyshev's inequality this means that $c_2$ is almost surely positive in this model. Numerical simulations imply that $(c_2-E[c_2])/n^2$ converges to a continuous distribution. See Figure~\ref{fig_hist_models}.

While in both models a knot projects to a star diagram, we point out a basic difference between the two. In the Petaluma model the original knot can always be realized by a polygon with $2(2n+1)$ segments, whereas the star model usually requires many more segments, possibly as many as~$n^2$. The typical length of those segments is tiny compared to the size of the knot. This resembles the random models based on a random walk in $\R^3$, that also have small edges. Another model with long edges is the grid model that we now describe.

A \emph{grid diagram} of order $m$ is a polygonal knot diagram consisting of $m$ vertical and $m$ horizontal edges, where vertical edges always pass over horizontal edges~\cite{cromwell1998arc, brunn1897uber}. The $x,y$ coordinates of the vertices are determined by a pair of permutations $\sigma,\pi \in S_m$ in the following way: $\sigma(0),\pi(0)$ $\;\to\;$ $\sigma(0),\pi(1)$ $\;\to\;$ $\sigma(1),\pi(1)$ $\;\to\;$ $\sigma(1),\pi(2)$ $\;\to\;\cdots\;\to\;$ $\sigma(m-1),\pi(0)$ $\;\to\;$ $\sigma(0),\pi(0)$. A random knot in the \emph{grid model} is obtained by picking $\sigma$ and $\pi$ independently and uniformly at random. 

Adams et al.~\cite[after Corollary 3.7]{adams2012knot} remark that a petal diagram can be turned into a grid diagram. The grid diagram has the same $\pi \in S_{2n+1}$ as in the petal diagram, and $\sigma$ defined by $\sigma(k)=nk \mod 2n+1$. 

This is demonstrated below for the trefoil knot, starting with a petal diagram with triangular petals and $\pi=(1,4,2,0,3)$. Such a diagram is the planar projection of a polygonal \emph{windmill knot}, where the straight lines through the center are lifted to horizontal segments at the appropriate heights. These segments are then folded at the center so that the other segments that connect them become vertical, which creates a \emph{watermill knot}. This is in fact a \emph{book knot} whose $2n+1$ pages are evenly spread out. Book knots can be represented by grid diagrams, whose horizontal lines come from pairs of segments to a vertical axis behind the diagram's plane.
\nopagebreak
\begin{center}
\begin{tabular}{ccccc}
&&&&\\
\tikz[thick,>=stealth]{
\foreach \angle in {36,108,...,360} \draw[line width = 1,decoration={markings,mark=at position 0.5 with {\arrow{>}}},postaction={decorate}] 
(0,0) -- +(\angle:1.1) -- +(\angle-36:1.1) -- (0,0); 
\foreach \angle/\num in {72/4, 144/3, 216/2, 288/1, 360/0} \draw (\angle-10:0.5) node[scale=0.7] {$\num$};
\pgfresetboundingbox
\clip (-1.25,-1.25) rectangle (1.25,1.25);
}
&\tikz[thick,>=stealth]{
\draw[lightgray, line width = 2] (0,0.15) -- (0,-0.45);
\def\angles{{125,180,235,-30,+30}};
\def\radii{{1.0,1.2,1.1,1.1,1.1}};
\foreach \xa/\ya/\xb/\yb in {1/0/4/0,0/1/3/1,4/2/2/2,3/3/1/3,2/4/0/4} { 
\draw[>=triangle 90 cap,white,<->,line width = 3,shorten >=1, shorten <=1] (0,\ya*0.15-0.45) +(\angles[\xa]-0:\radii[\xa]) -- +(\angles[\xa]+180:\radii[\xa]);
\draw[black,-,line width = 1] (0,\ya*0.15-0.45) +(\angles[\xa]-0:\radii[\xa]) -- +(\angles[\xa]+180:\radii[\xa]); 
\filldraw (0,\ya*0.15-0.45) circle[radius=0.025];}
\foreach \xa/\ya/\xb/\yb in {0/1/3/3,3/3/1/0,1/0/4/2,4/2/2/4,2/4/0/1} {
\draw[>=triangle 90 cap,white,<->,line width = 3,shorten >=2, shorten <=2] 
($ (0,\ya*0.15-0.45) + (\angles[\xa]+180:\radii[\xa]) $) -- ($ (0,\yb*0.15-0.45) + (\angles[\xb]:\radii[\xb]) $); 
\draw[black,-,line width = 1,decoration={markings,mark=at position 0.6 with {\arrow{<}}},postaction={decorate}] 
($ (0,\ya*0.15-0.45) + (\angles[\xa]+180:\radii[\xa]) $) -- ($ (0,\yb*0.15-0.45) + (\angles[\xb]:\radii[\xb]) $);}
\pgfresetboundingbox
\clip (-1.25,-1.25) rectangle (1.25,1.25);
}
&\tikz[thick,>=stealth]{
\draw[lightgray, line width = 2] (0,0.4) -- (0,-0.6);
\def\angles{{125,180,235,-30,+30}};
\def\radii{{1.0,1.2,1.0,1.1,1.1}};
\foreach \xa/\ya/\xb/\yb in {1/0/4/0,0/1/3/1,4/2/2/2,3/3/1/3,2/4/0/4} { 
\draw[>=triangle 90 cap,white,<->,line width = 3,shorten >=0, shorten <=2] (0,\ya*0.25-0.6) -- +(\angles[\xa]:\radii[\xa]);
\draw[>=triangle 90 cap,white,<->,line width = 3,shorten >=0, shorten <=2] (0,\ya*0.25-0.6) -- +(\angles[\xb]:\radii[\xb]); 
\draw[black,-,line width = 1] (0,\ya*0.25-0.6) -- +(\angles[\xa]:\radii[\xa]); 
\draw[black,-,line width = 1] (0,\ya*0.25-0.6) -- +(\angles[\xb]:\radii[\xb]);
\filldraw (0,\ya*0.25-0.6) circle[radius=0.03];}
\foreach \xa/\ya/\xb/\yb in {3/1/3/3,1/3/1/0,4/0/4/2,2/2/2/4,0/4/0/1} {
\draw[>=triangle 90 cap,white,<->,line width = 3,shorten >=2, shorten <=2] (0,\ya*0.25-0.6) ++(\angles[\xa]:\radii[\xa]) -- +(0,\yb*0.25-\ya*0.25); 
\draw[black,-,line width = 1,decoration={markings,mark=at position 0.5 with {\arrow{<}}},postaction={decorate}] 
(0,\ya*0.25-0.6) ++(\angles[\xa]:\radii[\xa]) -- +(0,\yb*0.25-\ya*0.25); }
\pgfresetboundingbox
\clip (-1.25,-1.25) rectangle (1.25,1.25);
}
&\tikz[thick,>=stealth]{
\draw[lightgray, line width = 2] (0.25,-0.8) -- (0.25,1.2);
\foreach \xa/\ya/\xb/\yb in {0/1/3/1,3/3/1/3,1/0/4/0,4/2/2/2,2/4/0/4} { 
\draw[black,-,line width = 1] (\xa*0.5-1,\ya*0.5-1) -- (0.25,\ya*0.5-0.8) -- (\xb*0.5-1,\yb*0.5-1); 
\filldraw (0.25,\ya*0.5-0.8) circle[radius=0.05];}
\foreach \xa/\ya/\xb/\yb in {3/1/3/3,1/3/1/0,4/0/4/2,2/2/2/4,0/4/0/1} { 
\draw[>=triangle 90 cap,white,<->,line width = 4,shorten >=1, shorten <=1] (\xa*0.5-1,\ya*0.5-1) -- (\xb*0.5-1,\yb*0.5-1); 
\draw[black,-,line width = 1,decoration={markings,mark=at position 0.35 with {\arrow{<}}},postaction={decorate}] (\xa*0.5-1,\ya*0.5-1) -- (\xb*0.5-1,\yb*0.5-1); }
\pgfresetboundingbox
\clip (-1.25,-1.25) rectangle (1.25,1.25);
}
&\tikz[thick,>=stealth]{
\foreach \xa/\ya/\xb/\yb/\d in {0/1/3/1/,3/3/1/3/,1/0/4/0/,4/2/2/2/,2/4/0/4/, 3/1/3/3/decorate,1/3/1/0/decorate,4/0/4/2/decorate,2/2/2/4/decorate,0/4/0/1/decorate} { 
\draw[>=triangle 90 cap,white,<->,line width = 4,shorten >=1, shorten <=1] (\xa*0.5-1,\ya*0.5-1) -- (\xb*0.5-1,\yb*0.5-1); 
\draw[black,-,line width = 1,decoration={markings,mark=at position 0.35 with {\arrow{<}}},postaction={\d}] (\xa*0.5-1,\ya*0.5-1) -- (\xb*0.5-1,\yb*0.5-1); }
\filldraw (-1,-0.5) circle[radius=0.05];
\pgfresetboundingbox
\clip (-1.25,-1.25) rectangle (1.25,1.25);
}
\\
Petaluma & Windmill & Watermill & Book & Grid \\&&&&
\end{tabular}
\end{center}

We created histograms for the Casson invariant of order-$m$ random knots in the grid model for $m=50,100,200$. As in the Petaluma model, these suggest that $c_2/m^2$ weakly converges. It would be interesting to extend the study of the $c_2$ moments to the grid model. While our methods seem to apply to this situation as well, the details are bound to be substantially more complicated.

Figure~\ref{fig_hist_models} displays numerically generated histograms of $c_2$ in the different models, normalized to have expectation $0$ and variance $1$. These seem to share certain properties. It is unknown but possible that there is some \emph{universal} family of distributions for several random models of knots.

\begin{figure}[htb]
\centering
\includegraphics[width=0.8\textwidth]{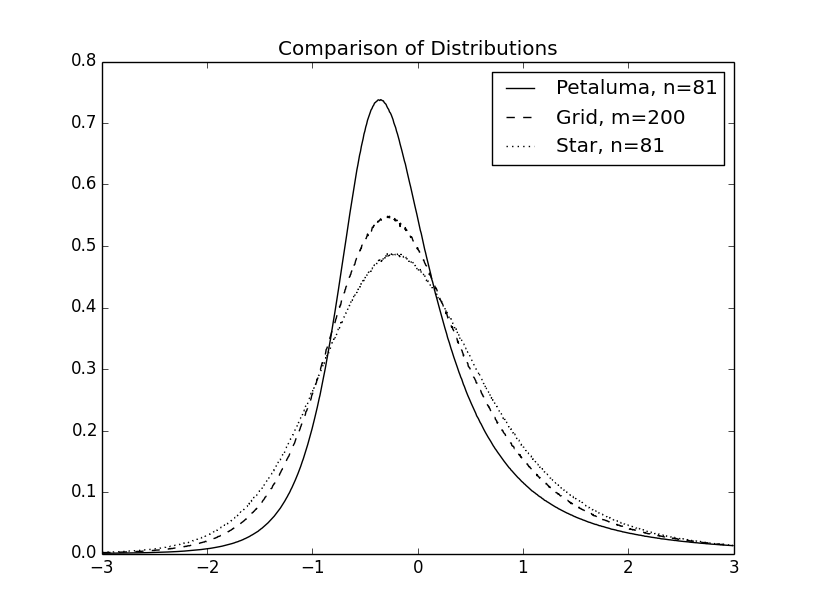}
\caption{Normalized distribution of $c_2$ in three models, based on over $10^8$ random samples.}
\label{fig_hist_models}
\end{figure}

\section{Order 3}\label{j3sect}

The knot invariant $v_3$ is the unique order-$3$ invariant that vanishes on the unknot, equals $1$ on the positive trefoil, and $-1$ on its reflection, the negative trefoil~\cite{polyak1994gauss}. Alternatively, $-6v_3$ is the third coefficient of the modified Jones polynomial, that is the power series in $h$ of the Jones polynomial $j(t)$ after the substitution $t=e^h$~\cite{chmutov2012introduction}. By  properties of the Jones polynomial, $v_3$ is antisymmetric with respect to reflection, and the distribution of $v_3(K_{2n+1})$ is symmetric around $0$. Hence, only even moments of $v_3$ are nonzero.

Here we work with a Gauss Diagram formula for $v_3$~\cite{polyak1994gauss}. A \emph{Gauss diagram} $D$ of a given knot diagram is the circle that maps to the knot diagram, with arrows connecting pairs of points that map to the same crossing. Each arrow is directed from the upper crossing point to the lower one, and marked with the sign of the crossing. We also mark the base point of the diagram, and orient the circle counterclockwise. The original knot diagram can be reconstructed from its Gauss diagram up to isotopy of the sphere $S^2$, though not all Gauss diagrams correspond to knots.

\begin{center}
\begin{tabular}{ccc}
\tikz[thick]{
\foreach \x in {0,1,...,4} {
\draw[>=triangle 90 cap,white,<->,line width = 6,shorten >=3, shorten <=3] (\x*288-54:2) -- (\x*288+90:2);
\draw[black,decoration={markings,mark=at position 0.25 with {\arrow{>}}},postaction={decorate}] (\x*288-54:2) -- (\x*288+90:2); }
\filldraw (306:2) circle (2pt);
\clip (-2,-2) rectangle (2,2);
} &
\tikz[thick]{
\draw[->,double,thick] (-1,0) -- (1,0);
\clip (-1,-2) rectangle (1,2);
} &
\tikz[thick]{
\filldraw (0,1.5) circle[radius=0.1];
\draw[line width=1.5pt] (0,0) circle[radius=1.5];
\foreach \start/\end/\sign in {290/110/+,330/140/-,0/190/+,215/40/+,70/265/+} \draw[>=triangle 60,->,line width=0.5pt] (0,0)+(\start:1.5) -- node[pos=-0.1,swap]{$\sign$} +(\end:1.5);
\clip (-2,-2) rectangle (2,2);
} \\
Knot Diagram & & Gauss Diagram
\end{tabular}
\end{center}

\newcommand{\AD}[1]
{\begin{tikzpicture}[baseline=-3pt,>=stealth]
\draw[thick] (0,0) circle[radius=0.25]; \filldraw (0,0.25) circle[radius=0.05];
\foreach \start/\end/\sign/\cosign in {#1} \draw[->] (\start:0.25) -- node[pos=-0.2,scale=0.6]{$\sign$} node[pos=1.2,scale=0.6]{$\cosign$} (\end:0.25);
\clip (-0.4,-0.4) rectangle (0.4,0.4);
\end{tikzpicture}}

A subdiagram of a Gauss diagram is obtained by considering a subset of the arrows. The number of appearances of~$D'$ as a subdiagram in $D$ is denoted $\left\langle D', D \right\rangle$. For example, if~$D$ is the Gauss diagram presented above, then
$$ \left\langle \;\AD{225/45/+/,315/135/+/}\;, D \right\rangle \;=\; 2 \;\;\;\;\;\;\;\; 
   \left\langle \;\AD{225/45/+/,315/135/-/}\;, D \right\rangle \;=\; 1 \;\;\;\;\;\;\;\; 
	 \left\langle \;\AD{225/45/-/,315/135/+/}\;, D \right\rangle \;=\; 0 \;\;\;\;\;\;\;\; 
	 \left\langle \;\AD{225/45/-/,315/135/-/}\;, D \right\rangle \;=\; 0 $$
A diagram without signs represents the formal sum over all ways to assign signs to the arrows, where each term is also multiplied by its signs. For example,
$$ \AD{225/45/,315/135/} \;=\; \AD{225/45/+/,315/135/+/}-\AD{225/45/-/,315/135/+/}-\AD{225/45/+/,315/135/-/}+\AD{225/45/-/,315/135/-/} \;\;.$$
The definition of $\left\langle \cdot , \cdot \right\rangle$ naturally extends to formal sums of diagrams.
In these terms, Lemma~\ref{c2formula} states
$$ c_2(K) \;=\; \left\langle \;\AD{225/45/,315/135/}\;, D \right\rangle $$
where $D$ is the Gauss diagram of any knot diagram that represents $K$. Note that in general such an expression depends on the choice of $D$. When independent of $D$, a formula of this form is called a \emph{Gauss diagram formula}. For example the Casson invariant of the trefoil knot shown above is
$$ \left\langle \;\AD{225/45/,315/135/}\;, D \right\rangle \;=\; 2 - 1 - 0 + 0 \;=\; 1 \;.$$
The order of a formula is the maximal number of arrows in one of its diagrams. Thus the above formula is of order~$2$. The Goussarov Theorem states that every finite type knot invariant of order~$m$ has a Gauss diagram formula of order~$m$~\cite{chmutov2012introduction,goussarov2000finite}. This formula is not unique. Already $c_2$ can be given also by
$$ \left\langle \;\AD{45/225/,135/315/}\;, D \right\rangle \;.$$
The following Gauss diagram formulas appear in the literature, the first three for $2v_3$, and the last one for $v_3$.
\begin{center}
\begin{tabular}{m{0.15\textwidth}m{0.65\textwidth}}
\hline \raggedright Polyak and Viro \cite{polyak1994gauss} &
$\AD{0/180/,250/110/,70/290/}+\AD{60/240/,310/170/,130/350/}+\AD{120/300/,10/230/,190/50/}+\AD{180/0/,70/290/,250/110/}+\AD{240/60/,130/350/,310/170/}+\AD{300/120/,190/50/,10/230/}+2\AD{0/180/,120/300/,240/60/}+2\AD{180/0/,300/120/,60/240/}$ \\ 
\hline \raggedright Willerton \cite{willerton1997vassiliev} &
$\AD{300/120/,230/10/,50/190/}+\AD{300/120/,10/230/,190/50/}+\AD{240/60/,170/310/,350/130/}+\AD{240/60/,310/170/,130/350/}
+\AD{0/180/,290/70/,250/110/}-\AD{180/0/,290/70/,250/110/}+\AD{180/0/,70/290/,250/110/}-\AD{0/180/,290/70/,110/250/} \hspace*{\fill}\linebreak
+2\AD{180/0/,290/70/,110/250/}+2\AD{0/180/,120/300/,240/60/}+2\AD{180/0/,300/120/,60/240/}$ \\
\hline \raggedright Goussarov, Polyak and Viro \cite{goussarov2000finite} &
$\AD{0/180/,120/300/,60/240/}+\AD{180/0/,300/120/,240/60/}+\AD{0/180/,120/300/,240/60/}+\AD{180/0/,300/120/,60/240/}
+\AD{120/300/,10/230/,50/190/}+\AD{300/120/,230/10/,190/50/}+\AD{60/240/,310/170/,130/350/}+\AD{240/60/,170/310/,350/130/} \hspace*{\fill}\linebreak
+\AD{0/180/,110/250/,290/70/}+\AD{180/0/,250/110/,70/290/}
+\AD{45/225/+/,135/315/+/}-\AD{45/225/+/,135/315/-/}+\AD{225/45/-/,315/135/+/}-\AD{225/45/-/,315/135/-/}$ \\
\hline \raggedright Chmutov and Polyak \cite{chmutov2012introduction} &
$\AD{180/0/,300/120/,240/60/}+\AD{0/180/,300/120/,240/60/}+\AD{240/60/,310/170/,350/130/}+\AD{0/180/,290/70/,250/110/}+\AD{300/120/,230/10/,190/50/}
+\AD{225/45/+/,315/135/+/}-\AD{225/45/-/,315/135/-/}$ \\
\hline
\end{tabular}
\end{center}
A typo in~\cite{goussarov2000finite} is corrected here. Note that our correction is different than in~\cite{chmutov2012introduction}. 

For the proof of Theorem~\ref{moments3} we adopt the formula by Goussarov, Polyak and Viro (GPV), which turns out to be best suited to generalize the $c_2$ arguments.

\subsection*{Proof of Theorem~\ref{moments3}}

The proof closely follows that of Theorem~\ref{moments2}, where we use the GPV formula for $v_3$. Hence we only highlight the adjustments that are required in each part of the proof.

\begin{enumerate}[leftmargin=12pt]
\item 
As for $c_2$, we represent $E[v_3^k]$ as a sum over patterns, parities and permutations. The only modification to be made is to extend the definition of a pattern from quadruples to \emph{arrow diagrams}, here simply meaning diagrams with no signs, as below. Consider the arrow diagrams that appear in the GPV formula:
$$ \mathcal{D} = \left\{ \AD{0/180/,120/300/,60/240/},\AD{180/0/,300/120/,240/60/},\AD{0/180/,120/300/,240/60/},\AD{180/0/,300/120/,60/240/},
\AD{120/300/,10/230/,50/190/},\AD{300/120/,230/10/,190/50/},\AD{60/240/,310/170/,130/350/},\AD{240/60/,170/310/,350/130/},
\AD{0/180/,110/250/,290/70/},\AD{180/0/,250/110/,70/290/},\AD{45/225/,135/315/},\AD{225/45/,315/135/}\right\} \;.$$
In contrast to the $c_2$ formula that uses only the last of these arrow diagrams, the $v_3$ formula is a combination of all twelve. Therefore, for the computation of $v_3^k$, a pattern records also which of the $12^k$ options for subdiagrams are involved in the $k$ corresponding terms in the sum. As for~$c_2$, it contains information about equalities and order relations between the relevant segments in the star diagram.

A \emph{pattern} $\bold{T}$ of order $k$ is a sequence $T_1,\dots,T_k$ of arrow diagrams in $\mathcal{D}$, whose arrow tips are marked with natural numbers. These numbers are non-decreasing when one moves counterclockwise from the base point, and their union is $\{1,\dots,t\}$ for some $t(\bold{T})$. For example, here is an order-$4$ pattern with $t=13$,
$$ \bold{T} = \left(\AD{180/0/6/10,300/120/9/2,240/60/6/11},\AD{300/120/5/1,230/10/5/8,190/50/5/9},\AD{225/45/7/13,315/135/9/3},\AD{240/60/6/12,170/310/4/9,350/130/9/2}\right) \;\;.$$
The segment numbers in an arrow diagram $T_i$ are denoted by $t_{i1},\dots,t_{i6}$ or $t_{i1},\dots,t_{i4}$. For example, here $t_{11}=2$ and $t_{12}=t_{13}=6$. The set of all order-$k$ patterns is denoted by $\mathcal{T}_k$.

\item
By the same reasoning as for $c_2$, we arrive at the following expression:
$$ E\left[v_3^k\right] \;=\; \sum\limits_{\bold{T} \in \mathcal{T}_k} \sum\limits_{\varepsilon \in \{\pm\}^t} \sum\limits_{\sigma \in S_t} \frac{1}{t!} \binom{n + z(\varepsilon)}{t} \;\prod\limits_{i=1}^k c(T_i,\varepsilon) \cdot \mathcal{I}\left[\substack{\sigma(t_{il}) \;>\; \sigma(t_{im}) \\ \text{for every arrow }t_{il} \to t_{im}\text{ in }T_i}\right] .$$
Here the function $c(T_i,\varepsilon)$, defined below, assumes the role of $-f(T_i,\varepsilon)\varepsilon(t_{i1})\varepsilon(t_{i2})\varepsilon(t_{i3})\varepsilon(t_{i4})$ in the proof of Theorem~\ref{moments2}. 

We first show that $\varepsilon$ determines the signs in an arrow diagram $T_i$, and yields a Gauss diagram. Recall from the discussion following Proposition~\ref{order} that the sign of a crossing point in a star diagram can be recovered from the parities of the two crossing segments. Specifically, every arrow $t_{il} \to t_{im}$ in $T_i$ is signed by $\pm\varepsilon(t_{il})\varepsilon(t_{im})$ depending on whether the crossing is ascending or descending.

The function $c(T_i,\varepsilon)$ is then defined to be the coefficient of this Gauss diagram in the GPV formula or $0$. This depends on whether the numbers at the arrow tips can or cannot correspond to a choice of segments as we traverse the star diagram, with crossings as in the Gauss diagram. The conditions are 
\begin{itemize}
\item 
Since a segment doesn't cross itself and two segments cross at most once, no arrow can point from a number to itself, and no two arrows connect the same pair of numbers.
\item 
If several arrow tips share a segment number then their order should agree with the order induced from the parities $\varepsilon$ by means of Proposition~\ref{order}.
\end{itemize}

For example, in the above $\bold{T}$ the compatibility conditions are $\varepsilon(6) = \varepsilon(10) \neq \varepsilon(11)$ for $T_1$, $\varepsilon(1) = \varepsilon(5) = \varepsilon(8) = \varepsilon(9)$ for $T_2$, and $\varepsilon(2) = \varepsilon(4)$ or $\varepsilon(4) \neq \varepsilon(9)$ for $T_4$. Note also that $c(T_3,\varepsilon)$ must be $0$ if $\varepsilon(7) = \varepsilon(13)$, because the sign of the arrow is given by $\varepsilon(7)\varepsilon(13)$ for this ascending crossing, while this arrow appears only with a minus sign in the two last diagrams of the GPV formula.

\item
The crucial point in Lemma~\ref{masstransfer} is deriving an expression where all terms are at least quadratic in the $\varepsilon(i)$'s. The following lemma plays the analogous role in the $v_3$ case.

Let $J \subseteq [t]$, and denote by $\mathcal{U}_J$ be the set of all numbered arrow diagrams from any pattern in~$\mathcal{T}_k$, that are marked exactly with the numbers in $J$. We represent their contribution to $v_3$ in the Fourier basis:
$$ \sum\limits_{S \in \mathcal{U}_J} c(S,\varepsilon) \cdot \mathcal{I}\left[\substack{\sigma(s_l) \;>\; \sigma(s_m) \\ \text{for every arrow }s_l \to s_m\text{ in }S}\right] \;=\;
\sum\limits_{I \subseteq J} \hat{c}(I,J,\sigma) \cdot\chi(I)(\varepsilon) $$
\begin{lemma}\label{j3masstransfer}
Let $J \subseteq [t]$ and $\sigma \in S_t$. Then $\hat{c}(I,J,\sigma) = 0$ in each of the following cases.
\begin{align*}
&|J| =6, \;|I| < 6 \\
&|J| =5, \;|I| < 4 \\
&|J| =4, \;|I| < 2
\end{align*}
\end{lemma}
\begin{proof}
This was checked by a computer program. It is sufficient to consider $J = \{1,\dots,j\}$ for $j \leq 6$ and all $\sigma \in S_j$. 

A~priori, there are $10$ arrow diagrams if $j=6$, $50$ for $j=5$ since each number in $J$ may repeat twice, and $102$ for $j=4$: $60$ in which two numbers appear twice, $40$ with one that repeats three times, and the two $2$-arrow diagrams. The condition on $\sigma$ leaves us with some subset of those diagrams. Then we compute the discrete Fourier transform of the remaining sum as a function of $\varepsilon$, and assert that the appropriate low order coefficients vanish. The verification program can be found in the supplementary material~\cite{supplementarymaterial}.
\end{proof}

Given a pattern $\bold{T} \in \mathcal{T}_k$, we denote by $J_i \subseteq [t]$, the set of numbers that appear in the diagram $T_i$. Note that $3 \leq |J_i| \leq 6$ and $\bigcup_iJ_i=[t]$ for some $t \leq 6k$. We rewrite the $k$th moment as a sum over all such sequences of sets.
$$ E\left[v_3^k\right] \;=\; \sum\limits_{J_1,\dots,J_k} \;\sum\limits_{\substack{I_1,\dots,I_k \\ I_i \subseteq J_i}} \;\sum\limits_{\sigma \in S_t} \;\sum\limits_{\varepsilon \in \{\pm\}^t} \frac{1}{t!} \binom{n + z(\varepsilon)}{t} \;\prod\limits_{i=1}^k \;\hat{c}(I_i,J_i,\sigma) \;\chi(I_i)(\varepsilon) .$$

\item
As in the proof of $c_2$, we view the $\varepsilon$ sum as an inner product in $\mathcal{W}_t$ between $\tbinom{n+z}{t} = \sum_{r=0}^{t} \tbinom{n}{t-r} \tbinom{z}{r}$ and $\prod_i\hat{c}(I_i,J_i,\sigma)\chi(I_i)$. The $r$th summand of the former is in $\text{\emph{span}}\{\chi(I) : |I| \leq 2r\}$, and has order $n^{t-r}$. We need terms with $t-r>3k$ to vanish, so it remains to show that $\prod_i\hat{c}(I_i,J_i,\sigma)\chi(I_i)$ is in $\text{\emph{span}}\{\chi(I) : |I| > 2r\}$.

Suppose that all $J_i$'s are disjoint. By Lemma~\ref{j3masstransfer}, if $\prod_i\hat{c}(I_i,J_i,\sigma) \neq 0$ then 
$$ \deg\prod\limits_{i=1}^k\chi(I_i) \;=\; \sum\limits_{i=1}^k |I_i| \;\geq\; \sum\limits_{i=1}^k (2|J_i|-6) \;=\; 2\left|\bigcup\limits_{i=1}^k J_i\right| - 6k \;=\; 2(t - 3k) \;>\; 2r \;\;. $$
If there exists a single common $j \in J_i \cap J_{i'}$, then $|\bigcup_i J_i|$ decreases by one and the degree on the left may decrease by two, since $\varepsilon(j)^2=1$.
By iterating this argument, the degree of $\prod_i\chi(I_i)$ always remains $>2r$.\qed
\end{enumerate}

For example, the second moment of $v_3$ is a polynomial of degree $6$. We compute the entire distribution of $v_3(K_{2n+1})$ for every $0 \leq n \leq 7$, and obtain
$$ E[v_3^2] \;=\; \frac{9298n^6-1101n^5-7145n^4+2175n^3-1433n^2-1794n}{5443200} \;.$$

We note that Lemma~\ref{j3masstransfer} would fail for the other three formulas for $v_3$. This may be related to the fact that the GPV formula extends to an invariant of virtual knots. We can partly see this relation. The case $|J|=6$ of the lemma follows from the fact that the coefficient of a maximum order term in a Gauss diagram formula of a virtual knot invariant is multiplicative at the signs of the arrows. The case $|J|=5$ also follows from this property together with the $6$-term relation for such formulas. 

An interesting question is whether it is possible to extend Lemma~\ref{j3masstransfer} and hence Theorem~\ref{moments3} to every Gauss diagram formula of a virtual knot invariant. It is conjectured that every finite type invariant of classical knots is induced from such a formula~\cite{goussarov2000finite}.

\bibliographystyle{abbrv}
{ \small
\bibliography{petaluma}
}

\end{document}